\documentclass[9pt, a4paper, twosided]{article}
\usepackage{bbm}
\usepackage{pifont}
\usepackage{bbding}
\usepackage{mathrsfs}
\usepackage{pgf,tikz}
\usepackage{amsfonts,amssymb,amsmath,indentfirst,amsthm}
\usepackage{epsfig}
\usepackage{fancyhdr}
\usepackage{graphicx}
\usepackage{cite}
\usepackage{url}
\usepackage{xcolor}
\usepackage{cite}
\usepackage{calc}
\usepackage{cases}
\usetikzlibrary{patterns}

\def\ess~inf{\mathop{\rm ess~inf}}

\numberwithin{equation}{section}

\newenvironment{key words}{\emph{\texttt{Keywords}}\mbox{  }}{ }
\newtheorem{theorem}{Theorem}[section]
\newtheorem{lemma}[theorem]{Lemma}

\newtheorem{proposition}[theorem]{Proposition}

\renewenvironment{proof}{\noindent{\textbf{Proof.}}}{\hfill$\Box$}
\theoremstyle{remark}
\newtheorem{remark}[theorem]{\textbf{Remark}}

\theoremstyle{plain}

\makeatletter

\newcommand{\Rmnum}[1]{\expandafter\@slowromancap\romannumeral #1@}
\makeatother

\usepackage{comment}

\allowdisplaybreaks

\begin{document}

\fancyhf{}

\fancyhead[EC]{W. Li and H. Wang}

\fancyhead[EL]{\thepage}

\fancyhead[OC]{Maximal functions along finite type curves}

\fancyhead[OR]{\thepage}

\renewcommand{\headrulewidth}{0.5pt}
\renewcommand{\thefootnote}{\fnsymbol {footnote}}
\title
{\textbf{On $L^{p}$-improving bounds for maximal operators associated with curves  of finite type  in the plane} \thanks{This work is supported by Natural Science Foundation of China (No.12301113; No.12271435).}}

\author{Wenjuan Li \\
 \small{School of Mathematics and Statistics, Northwestern Polytechnical University,
 Xi'an, 710129, China}  \\
 Huiju Wang \thanks{Corresponding author's email address: huijuwang@mail.nwpu.edu.cn. }\\
\small{School of Mathematics and Statistics, Henan University, Kaifeng, 475000, China}  }

\date{}
 \maketitle

{\bf Abstract:} In this paper,  we study the $L^{p}$-improving property for the maximal operators along a large class of curves of finite type in the plane with dilation set $E \subset [1,2]$. The $L^{p}$-improving region depends on the order of finite type and the fractal dimension of $E$. In particular, various impacts of non-isotropic dilations are also  considered.



\section{Introduction}

The study on maximal averages over dilated submanifolds has a long history. The ingredients that impact on the bounded exponent include, but are not limited to, type of dilation family (single parameter or not, isotropic dilation or not), dilation set (positive axis, compact interval or fractal set),
submanifold (homogeneity or not,  through the origin or not, non-vanishing Gaussian curvature or not),
measure (weighted measure or not, fractal measure or not). For more detailed information, one can refer to some articles, such as Iosevich-Sawyer \cite{ios2}, Ikoromov-Kempe-M\"{u}ller \cite{IKMU},  Roos-Seeger\cite{Roos}, etc.

Until now, the maximal functions along the sphere have been extensively studied, which is mainly because of the key geometric feature of the sphere: non-vanishing Gaussian curvature everywhere.
Let $n\geq 2$. Given a set $E\subset (0,+\infty)$,
we define the spherical maximal functions with dilation set $E$   by
\begin{equation*}
\mathcal{M}^E_{\mathbb{S}^{n-1}}f(y)=\sup_{t\in E} \biggl|\int_{\mathbb{S}^{n-1}}f(y-tx)d\sigma(x)\biggl|, \hspace{0.2cm}y\in \mathbb{R}^n,
\end{equation*}
here $d\sigma$ denotes the normalized surface measure on the unit sphere $\mathbb{S}^{n-1}$. We introduce some known results without endpoints related to $\mathcal{M}^E_{\mathbb{S}^{n-1}}$ in Case(1)-Case(5) below.\\
\textbf{Case(1)} $E=$\{Single point\}. Strichartz \cite{Strichartz} established  the $L^p\rightarrow L^q$ estimate for the single averaging operator  if $(\frac{1}{p},\frac{1}{q})$ belongs to the closed triangle with corners $(0,0)$, $(1,1)$ and $(\frac{n}{n+1},\frac{1}{n+1})$. \\
\textbf{Case(2)} $E=(0,+\infty)$. Stein \cite{stein3} ($n\geq 3$) and Bourgain \cite{JB} ($n=2$) successively proved that the global spherical maximal operator is $L^p\rightarrow L^p$ bounded if and only if $p>\frac{n}{n-1}$. For $n=2$, Mockenhaupt-Seeger-Sogge \cite{mss2} gave an alternative proof.\\
\textbf{Case(3)} $E=[1,2]$. Let $\mathcal{Q}$ be the closed quadrangle with corners $P_1=(0,0)$, $P_2=(\frac{n-1}{n},\frac{n-1}{n})$, $P_3=(\frac{n-1}{n},\frac{1}{n})$ and $P_4=(\frac{n(n-1)}{n^2+1},\frac{n-1}{n^2+1})$. In \cite{WS1,WS2}, Schlag and Sogge  proved that the corresponding local maximal operator is  $L^p\rightarrow L^q$ bounded when $(\frac{1}{p},\frac{1}{q})$ lies in the interior of the closed quadrangle $\mathcal{Q}$, and it fails when  $(\frac{1}{p},\frac{1}{q})$ belongs to the exterior. More related results can be found in \cite{LWZ}.\\
\textbf{Case(4)} $E\subset(0,\infty)$. For $\delta>0$, let $N(E,\delta)$ be the minimal number of $\delta$-intervals needed to cover $E$, and  define $\kappa(E)=\overline{\lim}_{\delta\rightarrow 0}[\sup_k h(E\cap I_k, \delta)]$, where $I_k=[2^k,2^{k+1}]$ and $h(E,\delta)=\frac{\log N(E,\delta)}{-\log \delta}$. Seeger-Wainger-Wright \cite{Seeger2} characterized  the $L^p$-boundedness of the local maximal operator $\mathcal{M}^E_{\mathbb{S}^{n-1}}$ up to the endpoint $p=1+(n-1)^{-1}\kappa(E)$.\\
\textbf{Case(5)} $E\subset[1,2]$. Denote by $\beta$ the (upper) Minkowski dimension of $E$, and by  $\gamma$ the (upper) Assouad dimension of $E$. Let $\mathcal{Q}^n(\beta,\gamma)$ be the closed convex hull of the points
\[Q_1=(0,0), \quad \quad Q_2=(\frac{n-1}{n-1+\beta},\frac{n-1}{n-1+\beta}),\]
\[Q_3=(\frac{n-\beta}{n-\beta+1},\frac{1}{n-\beta+1}), \quad \quad Q_4=(\frac{n(n-1)}{n^2+2\gamma-1},\frac{n-1}{n^2+2\gamma-1}).\]
Denote by $\mathcal{R}^n(\beta,\gamma)$ the union of the interior of $\mathcal{Q}^n(\beta,\gamma)$  with the line segment connecting $Q_1$ and $Q_2$, including $Q_1$, but excluding $Q_2$. For $n\geq 3$, Anderson-Hughes-Roos-Seeger \cite{AHRS} proved that if $(\frac{1}{p},\frac{1}{q})\in \mathcal{R}^n(\beta,\gamma)$, then $\mathcal{M}^E_{\mathbb{S}^{n-1}}$ is $L^p\rightarrow L^q$ bounded. Their results also include the case for $\gamma\leq 1/2$ in two dimension. For the more difficult case $n=2$ and $\gamma>1/2$,  Roos-Seeger \cite{Roos} demonstrated its validity.

For maximal functions along curves or surfaces of finite type, where the Gaussian curvature is allowed to vanish at some points, the corresponding results in Case(1) and Case(2) can be found in \cite{IKMU,I,ios1,ios2,ios3}. Recently, Li-Wang-Zhai \cite{LWZ} generalized some spherical results with isotropic dilations from Case(3) to more general curves or surfaces, and considered how non-isotropic dilations influence the boundedness for these local maximal operators. Moreover, they also obtained the sparse domination and weighted estimates for the corresponding global maximal operators.

However, to our best knowledge, the corresponding results for maximal functions along curves or surfaces of finite type still remain open in Case(4) and Case(5). In this paper, we concentrate  on Case(5), and characterize the $L^{p}$-improving property of the maximal function along curves of finite type in the plane over the dilation set $E \subset [1,2]$. We also study the various impacts of non-isotropic dilations. As a corollary, we get the sparse domination of the corresponding maximal functions over the dilation set $\{2^{k}t: t \in E \subset [1,2], k \in \mathbb{Z}\}$.

In order to describe our results, we first introduce some notations. We define the following regions of boundedness exponents that will be referred to later on:
\begin{equation}
\Delta:= \left\{(\frac{1}{p},\frac{1}{q}) \in \mathcal{R}^2(\beta,\gamma): \frac{d+1}{p} - \frac{d+1}{q} -1<0 \right\},
\end{equation}
and
\begin{align}
\widetilde{\Delta} := &\biggl\{(\frac{1}{p},\frac{1}{q}) \in \mathcal{R}^2(\beta,\gamma): \frac{1}{q} > \frac{d+1}{d-d\beta+1}\bigl(\frac{1}{p} -\frac{1}{d+1} \bigl) ,\nonumber \\
& \hspace{0.2cm} \textmd{and} \hspace{0.2cm} \frac{1}{q}> \frac{d \gamma + \beta}{2(d \gamma + \beta)-\beta(d+1+d \gamma) }\biggl( \frac{1}{p} -\frac{\beta}{d\gamma + \beta} \biggl) \biggl\}.
\end{align}
Here, $\beta \in (0,1)$ and $\gamma \in (0,1)$ are the upper Minkowski dimension and the upper Assouad dimension of $E$,  respectively. We recall their definitions for the reader's convenience.  For a set $E\subset \mathbb{R}$ and $\delta>0$, we denote by $N(E,\delta)$ the minimal number of compact intervals of length $\delta$ needed to cover $E$. The upper Minkowski dimension of a compact set $E$ is the smallest $\beta$ so that there is an estimate
\begin{equation}\label{Minkowski}
N(E,\delta)\leq C(\epsilon)\delta^{-\beta-\epsilon}
 \end{equation}
for all $0< \delta <1$ and $\epsilon>0$. The upper Assouad dimension is the smallest number $\gamma$ so that there exists $\delta_0>0$, and a constant $C_{\epsilon}$ for all $\epsilon>0$ such that for all $\delta\in (0,\delta_0)$ and all intervals $I$ of length $|I|\in(\delta,\delta_0)$ we have
\begin{equation}\label{Assouad}
N(E\cap I,\delta)\leq C_{\epsilon}(\delta/|I|)^{-\gamma-\epsilon}.
 \end{equation}
The more details can be found in \cite{AHRS} and \cite{Roos}.

Now we list our main results  as follows.

Let $\phi \in C^{\infty}(I,\mathbb{R})$, where $I$ is a bounded interval containing the origin, and
\begin{equation}\label{phi}
\phi(0)\neq 0; \hspace{0.1cm}\phi^{(j)}(0)= 0, \hspace{0.1cm}j=1,2,\cdots,m-1; \hspace{0.1cm}\phi^{(m)}(0)\neq 0 \hspace{0.1cm}(m\geq 1).
 \end{equation}
We study the $L^p \rightarrow L^{q}$ estimates for  local maximal functions with non-isotropic dilation $\delta_{t}(y) = (t^{a_{1}}y_1, t^{a_{2}}y_2)$ along the curves $(x,x^{d}\phi(x) +c)$ of finite type $d$ for $d \ge 2$, $d \in \mathbb{N}^{+}$ and $c \in \mathbb{R}$. Throughout this paper, $a_{1},a_{2}$ are positive real numbers.
Firstly, we consider the maximal operators when $d a_1 \neq a_2$.

\begin{theorem}\label{maintheorem1}
Define the averaging operator
\begin{equation}\label{averagefinite1}
A_{t}f(y):= \int_{\mathbb{R}} f(y_1-t^{a_{1}}x,y_2-t^{a_{2}}(x^d\phi(x) +c))\eta(x)dx, \quad da_{1} \neq a_{2},
\end{equation}
where $\eta$ is supported in a sufficiently small neighborhood of the origin. We can obtain the maximal inequality
$\|\mathop{sup}_{t\in E}|A_{t}|\|_{L^{p} \rightarrow L^{q} }\le C_{p,q}$
if one of the following conditions holds:

\textbf{(C1)} $c =0$ and $(\frac{1}{p},\frac{1}{q}) \in \Delta $;

\textbf{(C2)} $c \neq 0$  and $(\frac{1}{p},\frac{1}{q}) \in \widetilde{\Delta}$.

\end{theorem}

\begin{remark}
 From the proof of  Theorem \ref{maintheorem1}, we can observe that the corresponding $L^{p} \rightarrow L^{q}$  estimates hold for the maximal operator along the finite type curves $(x,x^{d}+c)$, $c \in \mathbb{R}$. In particular, we can  construct some counterexamples to show that the corresponding result is almost sharp, see Section
\ref{sharpresultsection} below.
\end{remark}

\begin{remark}
When $d =2$ and $\gamma\in [0,1/2]$, the result in Theorem \ref{maintheorem1} coincides with the 2-dimensional result in {\cite[Theorem 1]{AHRS}}. When $d =2$, $\gamma\in (1/2,1)$ and $\gamma$ equals to the quasi-Assouad dimension of $E$, the result in Theorem \ref{maintheorem1} coincides with the 2-dimensional result in {\cite[Theorem 1.1]{Roos}}. In addition, Theorem  \ref{maintheorem1} connects the corresponding results when $E=\{\text{Single point}\}$ and  $E=[1,2]$, which are obtained by Iosevich-Sawyer \cite{ios3} and Li-Wang-Zhai \cite{LWZ}, respectively.
\end{remark}

 The proof of  Theorem \ref{maintheorem1} was inspired by \cite{AHRS} and \cite{Roos}. We first decompose the support of $\eta$ dyadically and make good use of isometric transform on $L^{p}(\mathbb{R}^{2})$, so that in the corresponding dyadic operator the curve will be away from the flat point $x =0$. Then we decompose the Fourier side of $f$ into dyadic annulis, and define the operator $\widetilde{A_{t,j}^{k}}$, $k \gg 1$, by equality (\ref{decomposeoperator1}) and equality (\ref{decomposeoperator3}) for $j \ge 1$ and $j =0$, respectively. According to the condition $da_{1} \neq a_{2}$, a locally constant property established in Lemma \ref{locallyconstant} below implies that $\widetilde{|A_{t,j}^{k}}f|$ can be viewed as a constant in an interval of length $2^{-j}(c2^{dk}+1)^{-1}$. Next, we will establish an $L^2$-orthogonality estimate (Proposition \ref{keylemma1}) for a family of Fourier integral operators, whose phase function is different from that in {\cite[Proposition 4.2]{Roos}}.

\begin{center}
\includegraphics[height=8cm]{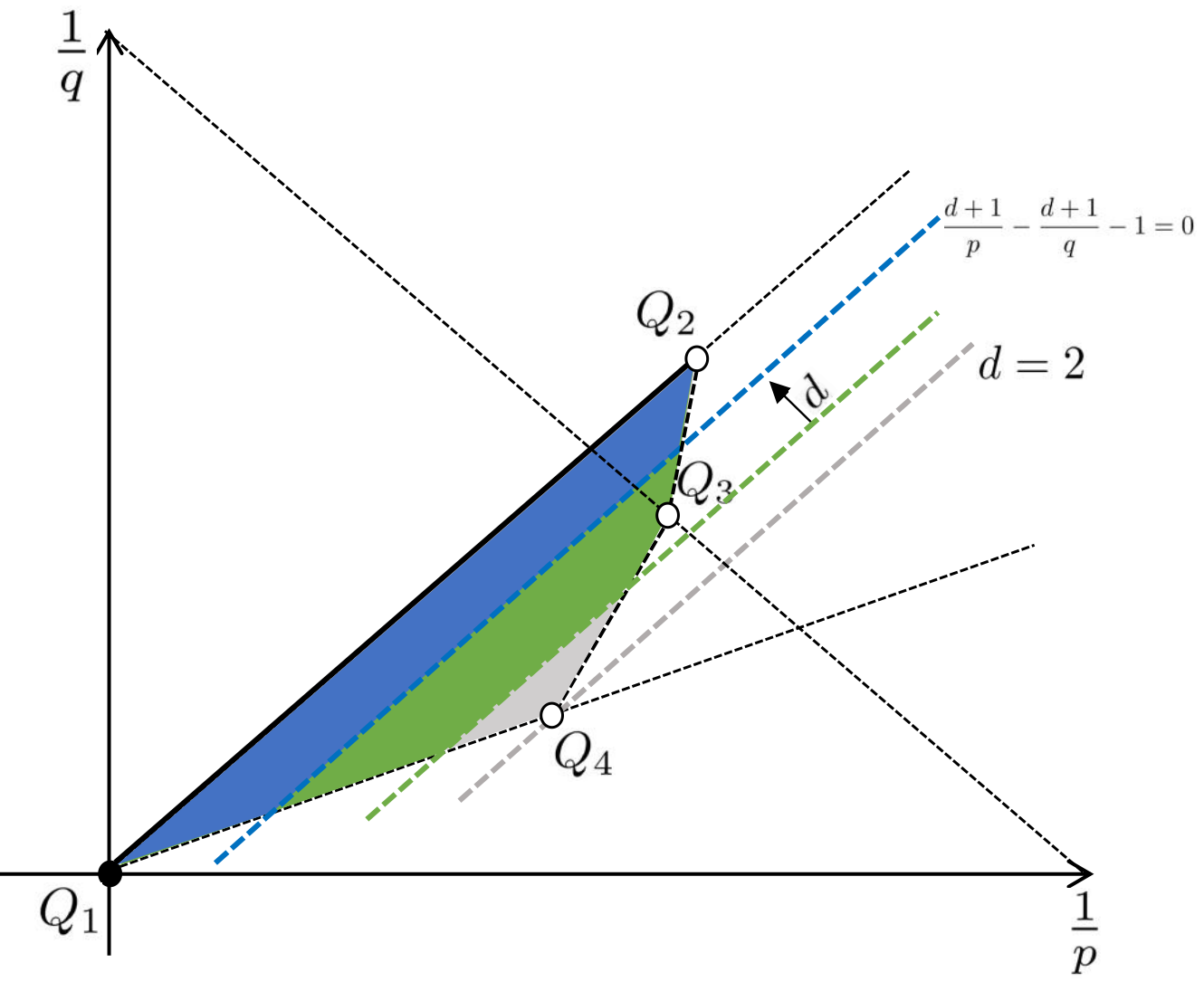}
\end{center}
\begin{center}
Figure 1. \textbf{(C1)} The region $\Delta$ of bounded exponents will  shrink  diagonally when $d$ becomes larger. If $d=2$, the bounded region $\Delta$ is the quadrilateral $\mathcal{R}^2(\beta,\gamma)$ (Grey+Green+Blue). As $d$ becomes larger, $\Delta$ turns to the pentagon (Green+Blue), then turns to the quadrilateral (Blue).
\end{center}

\begin{center}
\includegraphics[height=8cm]{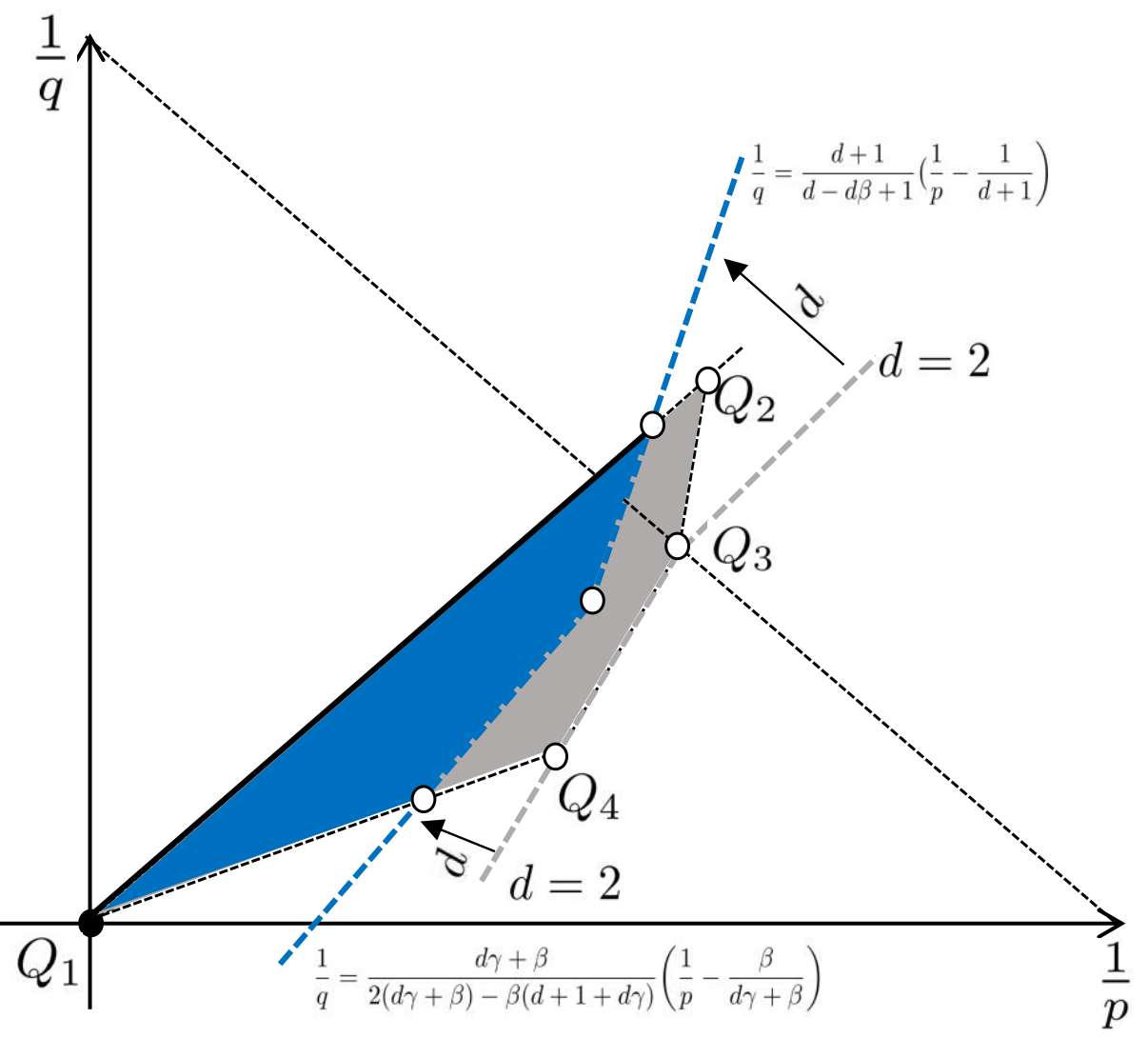}
\end{center}
\begin{center}
Figure 2. \textbf{(C2)} The region $\tilde{\Delta}$ of bounded exponents will  shrink when $d$ becomes larger. If $d=2$, the bounded region $\widetilde{\Delta}$ is the quadrilateral $\mathcal{R}^2(\beta,\gamma)$ (Grey+Blue). As $d$ becomes larger, $\widetilde{\Delta}$ turns to the quadrilateral (Blue).
\end{center}

 Now let's turn to consider the case when $da_{1} = a_{2}$.
 Similar to the observation in \cite{LWZ}, we know that non-isotropic dilations have a significant impact on the boundedness for the corresponding maximal operator along the curves of finite type. For example,  let $A_{t}$ be the averaging operator along the homogeneous curve $(x, x^{d})$ defined by
\begin{equation}\label{averagefinite2}
A_{t}f(y):= \int_{0}^{1}f(y_1-tx,y_2-t^dx^d)dx.
\end{equation}
The result from Iosevich-Sawyer \cite{ios3} shows that if the point  $(\frac{1}{p},\frac{1}{q})$ lies within the closed triangle with vertices $(0,0)$, $(1,1)$, $(2/3,1/3)$ and satisfies $\frac{d+1}{p}-\frac{d+1}{q}-1<0$, then there exists a   constant $C_{p,q} >0$ such that  $\|\mathop{sup}_{t\in E}|A_{t}|\|_{L^{p} \rightarrow L^{q} } \le C_{p,q} $. It is clear that the regularity of this result is better than that of the maximal operator considered in the above Theorem \ref{maintheorem1}.  Thus, one would expect that such a good level of regularity could be applied to more general curves. However, this regularity result  no longer holds for the curves $(x,x^{d}\phi(x))$ near the origin with dilation $\delta_{t}(y)=(ty_{1}, t^{d}y_{2})$,  which  can be considered as a small perturbation of the homogeneous curve in (\ref{averagefinite2}). To be more specific, we have obtained the following result for
the corresponding local maximal functions with dilation set $E \subset [1,2]$.

\begin{theorem}\label{hhisomaintheorem}
Define the averaging operator
\begin{equation}\label{averagefinite3}
A_{t}f(y):= \int_{\mathbb{R}} f(y_1-tx,y_2-t^{d} x^d\phi(x) )\eta(x)dx,
\end{equation}
where $\eta$ is supported in a sufficiently small neighborhood of the origin. Then for  $(\frac{1}{p},\frac{1}{q}) \in \Delta$, there exists a constant
$C_{p,q} >0$ such that $\|\mathop{sup}_{t\in E}|A_{t}|\|_{L^{p} \rightarrow L^{q} }\le C_{p,q} $.
\end{theorem}

 Comparing with Theorem \ref{maintheorem1}, the proof of Theorem \ref{hhisomaintheorem} requires a more involved argument. Indeed, after employing the stationary phase method, we need to consider a family of Fourier integral operators related to surfaces whose non-zero principal curvature is dominated by $2^{j-mk}$, $j\geq 1$ and $k$ is large enough,  one can see equality (\ref{phasefunction2}) below for the phase function of the Fourier integral operators. To obtain  the desired estimate for $\gamma \in (1/2,1)$, it is necessary to establish almost orthogonality  between different Fourier integral operators when the variables $t$ and $t^{\prime}$ are sufficiently separated. But this becomes  impossible if $2^{j-mk} \rightarrow 0$.  Taking into account this observation and the locally constant property discussed  in Remark \ref{localyconstant+} below, we consider the cases $j \le mk + \epsilon^{\prime} j$ and $j > mk + \epsilon^{\prime} j$, respectively, where $\epsilon^{\prime} $ can be sufficiently small.

 For $j \leq mk + \epsilon^{\prime} j$, no orthogonality applies in this case. Thanks to the corresponding results related to $\sup_{t \in [1,2]}|\widetilde{A_{t,j}^{k}}f|$ in \cite{LWZ}, please also refer to Lemma \ref{mainlemma1hom} in this paper, we are able to complete the proof of this part.

 For $j > mk + \epsilon^{\prime} j$, we  consider the Fourier integral operators related to surfaces with non-zero principal curvature not very small. However, in the  orthogonality argument, we have to carefully describe the influence caused by $2^{-mk}$. Fortunately, the influence from $2^{-mk}$ is annihilated since the locally constant property for $da_{1}=a_{2}$ enables us to divide the set $E$ by intervals of length $2^{-j+mk}$. The readers can see Proposition \ref{keylemma1+} below for more details.

\begin{remark}
The standard argument as in \cite{Lacey}, \cite{AHRS} and \cite{LWZ} implies sparse domination bounds for the corresponding global maximal operators defined by
\begin{equation*}
\sup_{k\in \mathbb{Z}}\sup_{t\in E}|A_{2^kt}f(y)|,\hspace{0.3cm} y\in \mathbb{R}^2,
\end{equation*}
where the averaging operator is defined by (\ref{averagefinite1}) and (\ref{averagefinite3}), respectively.
\end{remark}

\textbf{Conventions}: Throughout this article, we shall use the notation $A\ll B$, which means that  there is a sufficiently large constant $G$, which is much larger than $1$ and does not depend on the relevant parameters arising in the context in which
the quantities $A$ and $B$ appear, such that $G A\leq B$.  By
$A\lesssim B$ we mean that $A \le CB $ for some constant $C$ independent of the parameters related to  $A$ and $B$.

\section{Preliminaries}
In this section, we will prepare for the proofs of Theorem \ref{maintheorem1} and Theorem \ref{hhisomaintheorem} by introducing some notations and  useful lemmas. Additionally, we will also present some results that have been obtained in \cite{WL,LWZ}.

For the convenience of the proof, we always assume that $\eta$ is supported in the ball $\{x\in \mathbb{R}^2: |x|<\epsilon\}$. We choose  $\tilde{\rho}\in C_0^{\infty}(\mathbb{R})$ such that  supp $\tilde{\rho}\subset\{x: 1/2 \leq|x|\leq 2\}$  and $\sum_{k\in \mathbb{Z}}\tilde{\rho}(2^kx)=1$ for $x\in \mathbb{R}\backslash \{0\}$.
%By isometric transform  on $L^p(\mathbb{R}^2)$, we are reduced to the following averaging operator
%\begin{align*}
%\widetilde{A_t^{k}}f(y)	=\frac{1}{(2\pi)^2}\int_{{\mathbb{R}}^2} e^{i\xi\cdot y} e^{-it^{a_{2}}2^{dk}c \xi_{2}} \int_{\mathbb{R}}e^{-it^{a_{1}}\xi_1x-it^{a_{2}}\xi_{2}x^d\phi(\frac{x}{2^{k}})}\tilde{\rho}(x) dx\hat{f}(\xi)d\xi.
%\end{align*}
Meanwhile, we select a non-negative function $\psi \in C_0^{\infty}(\mathbb{R})$ such that supp $\psi \subset[1/2,2]$ and $\sum_{j\in\mathbb{Z}}\psi (2^{-j}r)=1$ for $r>0$. We define the dyadic operators
\begin{equation}\label{decomposeoperator1}
\widetilde{A_{t,j}^k}f(y)=\frac{1}{(2\pi)^2}\int_{{\mathbb{R}}^2}e^{i\xi\cdot y} e^{-it^{a_{2}}2^{dk}c \xi_{2}} \int_{\mathbb{R}}e^{-it^{a_{1}}\xi_1x-it^{a_{2}}\xi_{2}x^d\phi(\frac{x}{2^{k}})}\tilde{\rho}(x)\eta(2^{-k}x) dx\psi (2^{-j}|\delta_{t}\xi|)\hat{f}(\xi)d\xi
\end{equation}
for $j \ge 1$, and
\begin{equation}\label{decomposeoperator3}
\widetilde{A_{t,0}^k}f(y)=\frac{1}{(2\pi)^2}\int_{{\mathbb{R}}^2}e^{i\xi\cdot y} e^{-it^{a_{2}}2^{dk}c \xi_{2}} \int_{\mathbb{R}}e^{-it^{a_{1}}\xi_1x-it^{a_{2}}\xi_{2}x^d\phi(\frac{x}{2^{k}})}\tilde{\rho}(x)\eta(2^{-k}x) dx
\psi_{0}(|\delta_{t}\xi|)\hat{f}(\xi)d\xi,
\end{equation}
where $\psi_{0}(r) = \sum_{j \le 0} \psi(2^{-j}r)$.
%Then we  denote by $\widetilde{\mathcal{M}_{j}^k}$ and $\widetilde{\mathcal{M}^{k,0}}$ the corresponding maximal operator over $E$, respectively.

Firstly, we state the following observation:  when $da_{1} \neq a_{2}$, $|\widetilde{A_{t,j}^{k}}f|$ can be viewed as a constant when  $t$-variable changes in an interval of length $2^{-j}(c2^{dk} + 1 )^{-1}$. More precisely, we have  the following lemma.

\begin{lemma}\label{locallyconstant}
Let $I$ be an interval contained in $[1,2]$ with a length of $2^{-j}(c2^{dk}+1)^{-1}$.
If $da_{1} \neq a_{2}$, then for each  positive integer $N \gg 2$ and real numbers $q\ge p \ge 1$, there holds
\begin{align}\label{constantinequality}
\biggl\| \sup_{t \in I}|\widetilde{A_{t,j}^{k}}| \biggl\|_{L^{q}(\mathbb{R}^{2})} &\lesssim_{N} 2^{j}(c2^{dk} + 1) \int_{1/2}^{4} \frac{  \|\widetilde{A_{s,j}^{k}}f  \|_{L^{q}(\mathbb{R}^{2})} }{(1+2^{j}(c2^{dk}+1) dist(s,I))^{N}}ds \nonumber\\
&\quad \quad + 2^{-jN}(c2^{dk}+1)^{-N} \|f\|_{L^{p}(\mathbb{R}^{2})}.
 \end{align}
 Here, $dist(s,I)$ represents the distance from $s$-variable to the interval $I$.
\end{lemma}
\begin{proof}
Let $\rho$ and $\tilde{\psi}$ be bump functions that are supported in $[1/2,4]$ and $[-c_{1},c_{1}]$ respectively, where $c_{1}$ is a positive real number larger than 1 and depends on $d$, $a_1$, $a_2$, $\phi$. We decompose
\begin{align}\label{locdecompose}
\rho(t)\widetilde{A_{t,j}^{k}}f(y) &=\int_{\mathbb{R}}e^{i \tau t} \tilde{\psi}(2^{-j}(c2^{dk}+1)^{-1}\tau) \int_{\mathbb{R}}e^{-i\tau s} \rho(s) \widetilde{A_{s,j}^{k}}f(y)ds d\tau \nonumber\\
 \quad \quad &+ \int_{\mathbb{R}}e^{i \tau t} [1- \tilde{\psi}(2^{-j}(c2^{dk}+1)^{-1}\tau) ] \int_{\mathbb{R}}e^{-i\tau s} \rho(s) \widetilde{A_{s,j}^{k}}f(y)ds d\tau \nonumber\\
&:=I + II.
\end{align}
It is obvious that
\begin{equation}\label{LocEI}
|I| \le 2^{j}(c2^{dk} + 1) \int_{\mathbb{R}} \frac{|\rho(s)\widetilde{A_{s,j}^{k}}f(y)|}{(1+2^{j}(c2^{dk}+1)|t-s|)^{N}}ds.
\end{equation}
For the second term $II$, we can write
\begin{align}
II % &=\int_{\mathbb{R}^{2}}\int_{\mathbb{R}}e^{i \tau t} [1- \psi(2^{j}(c2^{dk}+1)\tau) ] \int_{\mathbb{R}^{2}}e^{i(y-z) \cdot \xi} \beta(2^{-j}|\xi|) \int_{\mathbb{R}}e^{-i\tau s- is^{\alpha_{2}}c2^{dk}\xi_{2} -is^{\alpha_{1}}x\xi_{1}- is^{\alpha_{2}}x^{d}\phi(\frac{x}{2^{k}})\xi_{2}} \nonumber\\
% &\quad \quad  \times \rho(s)\tilde{\rho}(x) dxds d\xi d\tau f(z)dz \nonumber\\
&=\int_{\mathbb{R}}\int_{\mathbb{R}^{2}}\int_{\mathbb{R}}e^{i \tau t} [1- \tilde{\psi}(2^{-j}(c2^{dk}+1)^{-1}\tau) ]\int_{\mathbb{R}^{2}}e^{i(y-z) \cdot \xi -ix\xi_{1}} \nonumber\\
&\quad\times \int_{\mathbb{R}}e^{-i\tau s- is^{a_{2}}c2^{dk}\xi_{2} - is^{a_{2}-da_{1}}x^{d}\phi(\frac{x}{2^{k}s^{a_{1}}})\xi_{2}} \frac{1}{s^{a_{1}}}\tilde{\rho}(\frac{x}{s^{a_{1}}})\rho(s)\eta(2^{-k}\frac{x}{s^{a_{1}}}) \psi(2^{-j}|\delta_{s} \xi|) ds d\xi d\tau f(z)dz dx. \nonumber
\end{align}
Integration by parts with respect to $s$ and $\xi$ separately implies that
\begin{align}\label{locEII}
|II| &\le 2^{2j} \int_{1/2}^{4} \int_{\mathbb{R}^{2}} \int_{\mathbb{R}}  [1- \tilde{\psi}(2^{-j}(c2^{dk}+1)^{-1}\tau) ] \frac{1}{|\tau|^{N+4}}d\tau  \frac{|f(z)|}{(1+2^j|y-z-(0,c2^{dk}s^{\alpha_{2}})|)^{2}} dz ds \nonumber\\
&\lesssim 2^{-jN}(c2^{dk}+1)^{-N} 2^{2j} \int_{1/2}^{4} \int_{\mathbb{R}^{2}} \frac{|f(z)|}{(1+2^j|y-z-(0,c2^{dk}s^{\alpha_{2}})|)^{2}} dz ds.
\end{align}
Noting that for any $s \in \mathbb{R}$ and $t \in I$, we have $|t-s| \ge dist(s,I)$. Therefore, combining inequalities (\ref{locdecompose})-(\ref{locEII}), we can obtain
\begin{align}
\biggl\| \sup_{t \in I}|\widetilde{A_{t,j}^{k}}| \biggl\|_{L^{q}(\mathbb{R}^{2})} &\lesssim_{N} 2^{j}(c2^{dk} + 1) \int_{1/2}^{4} \frac{   \| \widetilde{A_{s,j}^{k}}f  \|_{L^{q}(\mathbb{R}^{2})} }{(1+2^{j}(c2^{dk}+1) dist(s,I))^{N}}ds \nonumber\\
&\quad \quad + 2^{-jN}(c2^{dk}+1)^{-N} 2^{2j} \int_{1/2}^{4} \| (1+2^{j}|\cdot-(0,2^{dk}cs^{a_{2}})|)^{-2} * f\|_{L^{q}(\mathbb{R}^{2})}ds,
 \end{align}
which means that inequality (\ref{constantinequality}) holds true according to Young's inequality.
\end{proof}

Moreover, we have the following remark for the case $da_{1}=a_{2}$, which can be obtained by modifying the proof of Lemma \ref{locallyconstant}.

\begin{remark}\label{localyconstant+}
\emph{Notice that if $c = 0$ and $da_{1}=a_{2}$, $j-mk > \epsilon^{\prime}j$ for some $\epsilon^{\prime} >0$, we get
\begin{align}
\biggl\|\sup_{t\in I}|\widetilde{A_{t,j}^{k}}f| \biggl\|_{L^{q}(\mathbb{R}^{2})} &\lesssim_{N} 2^{j-mk} \int_{1/2}^{4} \frac{\|\widetilde{A_{s,j}^{k}}f\|_{L^{q}(\mathbb{R}^{2})} }{(1+2^{j-mk}dist(s,I))^{N}}ds  + 2^{-jN}   \|f\|_{L^{p}(\mathbb{R}^{2})}.
 \end{align}
 Hence, $|\widetilde{A_{t,j}^{k}}f(y)|$ can be viewed as a constant when $t$-variable is restricted in an interval with length no more than $2^{-j+mk}$ provided that $j-mk > \epsilon^{\prime}j$ for some $\epsilon^{\prime} >0$.}
\end{remark}

Secondly, we show some known results from \cite{WL,LWZ}.  For $k \ge \log\frac{1}{\epsilon}$ and $j=0$,  the kernel estimate {\cite[inequality (2.12)]{LWZ}} yields that
\begin{equation}\label{A1}
\biggl\|\widetilde{A_{t,0}^k}f \biggl\|_{L^{q}(\mathbb{R}^{2})} \lesssim \|f\|_{L^{p}(\mathbb{R}^{2})}, \quad q \ge p \ge 1.
\end{equation}
When $k \ge \log\frac{1}{\epsilon}$ and $j \ge 1$, in references \cite{WL,LWZ}, the method of stationary phase  was used to prove the following fact, the operator $\widetilde{A_{t,j}^{k}}$ can be decomposed as
\[\widetilde{A_{t,j}^{k}}f=A_{t,j}^{k}f+R_{t,j}^{k}f,\]
where the main contribution term
\begin{equation}\label{mainterm1}
A^k_{t,j}f(y):=\int_{{\mathbb{R}}^2}e^{i(\xi\cdot y -ct^{a_{2}}2^{dk}\xi_{2}- \tilde{\Phi}(t,\xi,\delta))}\chi_{k,d,m}(t^{a_{1}-a_{2}}\frac{\xi_1}{ \xi_2})
\frac{A_{k,d,m}( \delta_{t}\xi)}{(1+| \delta_{t}\xi|)^{1/2}}
\psi (2^{-j}| \delta_{t}\xi|)\hat{f}(\xi)d\xi.
\end{equation}
Here,
\[-\tilde{\Phi}(t,\xi,\delta)=(d-1)t^{\frac{da_{1}-a_{2}}{d-1}}\xi_{2} \biggl(-\frac{\xi_{1}}{d\xi_{2}}\biggl)^{\frac{d}{d-1}}-t^{\frac{(a_{1}-a_{2})(m+d)}{d-1} +a_{2} } \xi_{2}\frac{\delta^m\phi^{(m)}(0)}{ m!}
\biggl(-\frac{\xi_1}{d\xi_2}\biggl)^{\frac{d+m}{d-1}}+ R(t,\xi,\delta,d),
\]
and $\delta = 2^{-k}$, $R(t,\xi,\delta,d)$ is homogeneous of degree one in $\xi$ and has at least $m+1$ power of $\delta$ (without loss of generality, we may assume that $\phi(0)=1$), $\chi_{k,d,m}$ is a smooth function supported in $[c_{k,m},\widetilde{c_{k,m}}]$, for certain non-zero constants $c_{k,m}$ and $\widetilde{c_{k,m}}$ dependent only on $k$ and $m$. $A_{k,m,d}$ is a symbol of order zero in $\xi$ and $\{A_{k,m,d}(\delta_t\xi)\}$ is contained in a bounded subset of symbol of order zero.  Especially, when $da_{1}=a_{2}$, we have
\begin{equation}\label{phasefunction2}
-\tilde{\Phi}(t,\xi,\delta)=(d-1) \xi_{2} \biggl(-\frac{\xi_{1}}{d\xi_{2}}\biggl)^{\frac{d}{d-1}}-t^{-a_{1}m } \xi_{2}\frac{\delta^m\phi^{(m)}(0)}{ m!}
\biggl(-\frac{\xi_1}{d\xi_2}\biggl)^{\frac{d+m}{d-1}}+ R(t,\xi,\delta,d).
\end{equation}
For convenience, we denote
\[a_{j,k,t}(\xi)=2^{\frac{j}{2}}\chi_{k,d,m}(t^{a_{1}-a_{2}}\frac{\xi_1}{ \xi_2})
\frac{A_{k,d,m}( \delta_{t}\xi)}{(1+| \delta_{t}\xi|)^{1/2}}
\psi(2^{-j}| \delta_{t}\xi|) . \]
By symmetry,  we may just keep our eyes on the first quadrant of the $\xi$-plane throughout the proof of Theorem \ref{maintheorem1} and Theorem \ref{hhisomaintheorem}.

We present some estimates for operators $A_{t,j}^{k}$ and $R_{t,j}^{k}$ obtained in references \cite{WL} and \cite{LWZ}. When $k\geq \log\frac{1}{\epsilon}$ and $j\geq 1$, the remainder term $R_{t,j}^{k}$ satisfies
\begin{equation}\label{remainderterm}
 \| R_{t,j}^{k}f  \|_{L^{q}(\mathbb{R}^{2})} \lesssim_{N}  2^{-jN} \|f\|_{L^{p}(\mathbb{R}^{2})}, \quad q \ge p \ge 1, \quad t \in [1,2]
\end{equation}
for any positive integer $N$.

%By the Sobolev's embedding, we get
%\begin{equation}\label{remainderterm1}
%\biggl\|\sup_{t \in [1,2]} |R_{t,j}^{k}f | \biggl\|_{L^{q}(\mathbb{R}^{2})} \lesssim_{N} (1+c2^{\frac{dk}{q}})2^{-jN} \|f\|_{L^{p}(\mathbb{R}^{2})}, \quad q %\ge p \ge 1.
% \end{equation}

For the operator $A_{t,j}^{k}$, we have the following estimates. By Plancherel's theorem, for each $t \in [1,2]$,
\begin{equation}\label{L2}
\|A_{t,j}^kf\|_{L^{2}(\mathbb{R}^{2})} \le 2^{-j/2} \|f\|_{L^{2}(\mathbb{R}^{2})}.
\end{equation}
By {\cite[Lemma 2.3.1]{WL}}, for each $t \in [1,2]$, we have
\begin{equation}\label{Linfty}
\|A_{t,j}^kf\|_{L^{\infty}(\mathbb{R}^{2})} \lesssim \|f\|_{L^{\infty}(\mathbb{R}^{2})},
\end{equation}
and by Young's inequality, there holds
\begin{equation}\label{L1}
\|A_{t,j}^kf\|_{L^{1}(\mathbb{R}^{2})} \lesssim \|f\|_{L^{1}(\mathbb{R}^{2})}.
\end{equation}
Moreover, Lemma 2.6 in \cite{LWZ} and Young's inequality  imply that
\begin{equation}\label{Linfty1}
\|A_{t,j}^kf\|_{L^{\infty}(\mathbb{R}^{2})} \lesssim 2^{j}\|f\|_{L^{1}(\mathbb{R}^{2})}.
\end{equation}

Finally, for the sake of  simplicity in the proof, we  assume that
\begin{equation}\label{defofbeta}
\sup_{0<\delta<1}\delta^{\beta}N(E,\delta)<+\infty
\end{equation}
holds  for fixed $0<\beta<1$ and
\begin{equation}
\sup_{0<\delta<1}\sup_{\delta\leq |I|<1}(\delta/|I|)^{\gamma}N(E\cap I,\delta)<+\infty
\end{equation}
holds for fixed $0<\gamma<1$.
This assumption could not affect our proof, as we do not consider the endpoint estimates in this article.
%But this may not be true for more general set determined by the upper Minkowski dimension and the upper Assouad dimension (see inequalities (\ref{Minkowski}) %and (\ref{Assouad})). Hence, we will miss the endpoint estimates.
Furthermore, for the sake of convenience, we set
\[(1/p_{1},1/q_{1})=(0,0), \quad \quad (1/p_{2},1/q_{2})=(\frac{1}{1+\beta},\frac{1}{1+\beta}),\]
\[(1/p_{3},1/q_{3})=(\frac{2-\beta}{3-\beta},\frac{1}{3-\beta}), \quad \quad (1/p_{4},1/q_{4})=(\frac{2}{3+2\gamma},\frac{1}{3+2\gamma}).\]

\section{The case when $da_{1} \neq a_{2}$: Proof of Theorem \ref{maintheorem1}} \label{istropic}

%Before commencing the proof of Theorem \ref{maintheorem1}, we first introduce some notations.

Theorem \ref{maintheorem1} can be implied by the following two lemmas.
 \begin{lemma}\label{mainlemma1}
For each $j \ge 0$ and $k \ge \log (1/\epsilon)$, $1 \le i \le 3$ and $\beta \in (0,1)$, there holds
\begin{equation}\label{piqi}
\biggl\| \sup_{t \in E}|\widetilde{A_{t,j}^k}f| \biggl\|_{L^{q_{i}}(\mathbb{R}^{2})} \lesssim (c2^{\frac{d\beta k}{q_{i}}} + 1) \|f\|_{L^{p_{i}}(\mathbb{R}^{2})},
\end{equation}
where the implied constant does not depend on $j$ or $k$.
Moreover, there exists a constant $\epsilon_{1}>0$ depending on $\beta$ such that
\begin{equation}\label{l2l2}
\biggl\| \sup_{t \in E}|\widetilde{A_{t,j}^k}f| \biggl\|_{L^{2}(\mathbb{R}^{2})} \lesssim (c2^{\frac{d\beta k}{2}} + 1) 2^{-\epsilon_{1}j}\|f\|_{L^{2}(\mathbb{R}^{2})}.
\end{equation}
\end{lemma}

\begin{lemma}\label{mainlemma3}
When $\gamma \in (0,1)$,  for each $j \ge 0$ and $k \ge \log(1/\epsilon)$, there holds
\begin{equation}\label{weakestimate}
\biggl\|\sup_{t\in E}|\widetilde{A_{t,j}^{k}}f |\biggl\|_{L^{q_{4}}({\mathbb{R}}^2)} \lesssim_{\gamma} (c2^{\frac{d\gamma k}{3+2\gamma}} +1) \max\{j,1\} \|f\|_{L^{p_{4}}({\mathbb{R}}^2)},
\end{equation}
where the implied constant depends on $\gamma$, but never depends on $j$ and $k$.
\end{lemma}

Now we will explain how to utilize Lemma \ref{mainlemma1}  and Lemma \ref{mainlemma3} to obtain Theorem \ref{maintheorem1} through interpolation. In the case $c=0$, interpolating between the estimate in (\ref{l2l2})  with (\ref{piqi}) when $i=1,2,3$ respectively,  and combining with the interpolation between the estimate  in (\ref{l2l2}) and (\ref{weakestimate}), we obtain the existence of a constant $\epsilon(p,q)>0$ such that when $(\frac{1}{p},\frac{1}{q})$ belongs to the region $\mathcal{R}^{2}(\beta,\gamma) \backslash \{(0,0)\}$, there holds
\[\biggl\| \sup_{t \in E}|\widetilde{A_{t,j}^k}| \biggl\|_{L^{p}(\mathbb{R}^{2}) \rightarrow L^{q}(\mathbb{R}^{2})} \lesssim  2^{-\epsilon(p,q)j}.\]
Therefore, combining with the isometric transformation, when $(\frac{1}{p},\frac{1}{q})$ satisfies $\frac{d+1}{p}-\frac{d+1}{q}-1<0$, we have
\begin{align}
\biggl\|\sup_{t \in E}|A_{t}| \biggl\|_{L^{p}(\mathbb{R}^{2}) \rightarrow L^{q}(\mathbb{R}^{2})} &\lesssim \sum_{k \ge \log(1/\epsilon)} 2^{\frac{(d+1)k}{p}-\frac{(d+1)k}{q} -k} \sum_{j \ge 0} \biggl\|\sup_{t \in E}|\widetilde{A_{t,j}^{k}}| \biggl\|_{L^{p}(\mathbb{R}^{2}) \rightarrow L^{q}(\mathbb{R}^{2})} \nonumber\\
 &\lesssim \sum_{ k \ge \log(1/\epsilon)} 2^{\frac{(d+1)k}{p}-\frac{(d+1)k}{q}-k} \sum_{j \ge 0} 2^{-\epsilon(p,q)j} \lesssim 1,\nonumber
\end{align}
which implies Theorem \ref{maintheorem1} (1). Next we will prove the conclusion of Theorem \ref{maintheorem1} (2).
The interpolation between (\ref{l2l2}) and (\ref{piqi}) (for $i=1,2,3$) yields that, for any $j \ge 0$ and $(\frac{1}{p},\frac{1}{q}) \in \mathcal{R}^{2}(\beta,\gamma)\backslash \{(0,0)\}$ that satisfies $\frac{2-\beta}{q}> \frac{1}{p}$, there exists a $\epsilon(p,q)>0$ such that
\begin{equation}\label{inter1}
\biggl\|\sup_{t \in E}|\widetilde{A_{t,j}^{k}}| \biggl\|_{L^{q}(\mathbb{R}^{2})} \lesssim 2^{\frac{d\beta k}{q}} 2^{-\epsilon(p,q)j}\|f\|_{L^{p}(\mathbb{R}^{2})}.
\end{equation}
Then we have
\begin{equation}\label{cneq0e1}
\biggl\|\sup_{t\in E}|A_{t}|\biggl\|_{L^{p}(\mathbb{R}^{2})\rightarrow L^{q}(\mathbb{R}^{2})} \lesssim  \sum_{k \ge \log(1/\epsilon)}2^{\frac{(d+1)k}{p}-\frac{(d+1)k}{q}-k}  2^{\frac{d \beta k}{q}} \sum_{j \ge 0} 2^{-\epsilon(p,q)j} \lesssim 1
\end{equation}
provided that $(\frac{1}{p},\frac{1}{q}) \in \mathcal{R}^{2}(\beta,\gamma)\backslash \{(0,0)\}$, $\frac{2-\beta}{q} > \frac{1}{p}$ and $\frac{d+1}{p} -\frac{d-\beta d +1}{q}-1<0$.
Furthermore, the interpolation between (\ref{smallgamma}) and (\ref{piqi}) (when $i=1$) implies that, for each $j \ge 0$ and $q \ge 3+2\gamma$, there holds
\begin{equation}\label{inter2}
 \biggl\|\sup_{t \in E}|\widetilde{A_{t,j}^{k}}| \biggl\|_{L^{q}(\mathbb{R}^{2})} \lesssim 2^{\frac{d\gamma k}{q}} \max\{j,1\}^{\frac{3+ 2\gamma}{q}} \|f\|_{L^{\frac{q}{2}}(\mathbb{R}^{2})}.
 \end{equation}
 Interpolating between (\ref{inter1}) and (\ref{inter2}), and combining with the fact that $\beta \le \gamma$, there exists a constant $\epsilon^{\prime}(p,q)>0$ such that
\[ \biggl\|\sup_{t \in E}|\widetilde{A_{t,j}^{k}}| \biggl\|_{L^{q}(\mathbb{R}^{2})} \lesssim 2^{\frac{d\gamma k}{q}} 2^{-\epsilon^{\prime}(p,q)j} \|f\|_{L^{p}(\mathbb{R}^{2})} \]
 for any $j \ge 0$ and $(\frac{1}{p},\frac{1}{q}) \in \mathcal{R}^{2}(\beta,\gamma)\backslash \{(0,0)\}$ that satisfies $\frac{2-\beta}{q} \le \frac{1}{p} < \frac{2}{q}$. It follows that
\begin{equation}\label{cneq0e2}
\biggl\|\sup_{t\in E}|A_{t}|\biggl\|_{L^{p}(\mathbb{R}^{2})\rightarrow L^{q}(\mathbb{R}^{2})}  \lesssim \sum_{k \ge \log(1/\epsilon)}2^{\frac{(d+1)k}{p}-\frac{(d+1)k}{q}-k} 2^{\frac{d\gamma k}{q}} \sum_{j \ge 0} 2^{-\epsilon^{\prime}(p,q)j} \lesssim 1,
\end{equation}
whenever $(\frac{1}{p},\frac{1}{q}) \in \mathcal{R}^{2}(\beta,\gamma)\backslash \{(0,0)\}$, $\frac{2-\beta}{q} \le \frac{1}{p} < \frac{2}{q}$ and $\frac{d+1}{p} - \frac{d-\gamma d+1}{q} - 1 < 0$.  By interpolation between (\ref{cneq0e1}) and (\ref{cneq0e2}) we can finish the proof of Theorem \ref{maintheorem1} (2).

Now we start to prove Lemmas \ref{mainlemma1}-\ref{mainlemma3} in Subsection \ref{proofofmainlemma1}-\ref{proofofmainlemma3}, respectively.

\subsection{Proof of Lemma \ref{mainlemma1}}\label{proofofmainlemma1}
 For $i=1$, inequality (\ref{piqi}) can be obtained by (\ref{Linfty}) directly. Inequality (\ref{piqi}) for $2 \le i \le 3$ and inequality (\ref{l2l2}) follow from the next lemma.
\begin{lemma}\label{basicestimate}
We have the following four estimates. \\
\textbf{(I)} For $q \ge p \ge 1$, we have
\begin{align}
\biggl\|\sup_{t\in E} |\widetilde{A_{t,0}^{k}}f| \biggl\|_{L^{q}({\mathbb{R}}^2)} \lesssim (c2^{\frac{d\beta k}{q}} +1) \|f\|_{L^{p}({\mathbb{R}}^2)}.
\end{align}
\textbf{(II)} For each $j \ge 1$, there holds
\begin{equation}
\biggl\|\sup_{t\in E}|A_{t,j}^{k}f| \biggl\|_{L^{1}({\mathbb{R}}^2)}\lesssim (c2^{d \beta k} +1)2^{\beta j}\|f\|_{L^1({\mathbb{R}}^2)}.
\end{equation}
\textbf{(III)} For each $j \ge 1$, we get
\begin{equation}
\biggl\|\sup_{t\in E}|A_{t,j}^{k}f| \biggl\|_{L^{\infty}({\mathbb{R}}^2)}\lesssim 2^{j}\|f\|_{L^1({\mathbb{R}}^2)}.
\end{equation}
\textbf{(IV)} For each $j \ge 1$, we obtain
\begin{equation}
\biggl\|\sup_{t\in E}|A_{t,j}^{k}f| \biggl\|_{L^{2}({\mathbb{R}}^2)}\lesssim (c2^{\frac{d \beta k}{2}} +1)2^{\frac{\beta j}{2}} 2^{-\frac{j}{2}}\|f\|_{L^2({\mathbb{R}}^2)}.
\end{equation}
\end{lemma}

Indeed, by the interpolation between (II) and (IV), and combining the estimate (I) we get  inequality (\ref{piqi}) for the point $(1/p_{2}, 1/q_{2}) $ $= (1/(\beta +1),1/(\beta +1))$. Similarly, the inequality (\ref{piqi}) for  $(1/p_{3}, 1/q_{3}) = ((2-\beta)/(3-\beta),1/(3-\beta))$ follows from  (I) and the interpolation among the estimates (III), (IV). The inequality (\ref{l2l2}) can be directly obtained from the estimate (IV), since $\beta \in (0,1)$.

 Now let us prove Lemma \ref{basicestimate}.  Firstly, the estimate (I) follows from inequality (\ref{A1}) and Lemma \ref{locallyconstant}.
We will outline the application of Lemma \ref{locallyconstant} here. Lemma \ref{locallyconstant} will be repeatedly used in the subsequent proofs, and due to its similar application here, we will not elaborate on it further. In fact, combining inequality (\ref{A1}), inequality (\ref{proofofmainlemma3}) and inequality (\ref{constantinequality}) (with $j=0$), we obtain
\begin{align}
\biggl\| \sup_{t \in I}|\widetilde{A_{t,0}^{k}}| \biggl\|_{L^{q}(\mathbb{R}^{2})} &\lesssim_{N}  (c2^{dk} + 1) \int_{1/2}^{4} \frac{  \|\widetilde{A_{s,0}^{k}}f  \|_{L^{q}(\mathbb{R}^{2})} }{(1+ (c2^{dk}+1) dist(s,I))^{N}}ds \nonumber\\
&\quad \quad +  (c2^{dk}+1)^{-N} \|f\|_{L^{p}(\mathbb{R}^{2})} \nonumber\\
&\lesssim_{N} (c2^{dk} + 1) \int_{1/2}^{4}  (1+ (c2^{dk}+1) dist(s,I))^{-N}ds  \|f  \|_{L^{p}(\mathbb{R}^{2})} \nonumber\\
&\quad \quad +  (c2^{dk}+1)^{-N} \|f\|_{L^{p}(\mathbb{R}^{2})} \nonumber\\
&\lesssim_{N}  \|f\|_{L^{p}(\mathbb{R}^{2})}, \nonumber
 \end{align}
where $I$ is any interval contained in $[1,2]$ and of length $(c2^{dk}+1)^{-1}$. Therefore, combining inequality (\ref{defofbeta}), we can get the estimate (I).

Next we will prove the estimates (II)-(IV). Combining inequality (\ref{L2}) with Lemma \ref{locallyconstant}, we obtain the estimate (IV) in Lemma \ref{basicestimate}.
The  estimate (II) follows from Lemma \ref{locallyconstant} and inequality  (\ref{L1}). Moreover, inequality (\ref{Linfty1}) directly leads to (III) in Lemma \ref{basicestimate}. $\hfill\square$

%In order to prove the estimate (III) in Lemma \ref{basicestimate}, we need the following lemma about kernel estimate of $A_{t,j}^kf$.
%\begin{lemma}\label{kernel1}
%For each  $\lambda \gg 1$, $y \in \mathbb{R}^{2}$ and $|t| \sim 1$,  there holds a uniform estimate
%\begin{equation}\label{kerneldecay}
%\biggl|\int_{{\mathbb{R}}^2}e^{i\lambda (\eta \cdot y -t^{a_{2}}c2^{dk}\eta_{2}-\tilde{\Phi}(t,\eta,\delta))}a_{j,k,t}(2^{j} \eta) d\eta \biggl| \lesssim \lambda^{-1/2}.
%\end{equation}
%\end{lemma}
%We omit the proof since it can be obtained by modifying the proof of Lemma 2.6 in \cite{LWZ}. According to Lemma \ref{kernel1}, Young's inequality  implies
%\begin{equation}\label{Linfty1}
%\|A_{t,j}^kf\|_{L^{\infty}(\mathbb{R}^{2})} \lesssim 2^{j}\|f\|_{L^{1}(\mathbb{R}^{2})}.
%\end{equation}

\subsection{Proof of Lemma \ref{mainlemma3}}\label{proofofmainlemma3}
Before presenting the proof, we will first make some reductions. For each $j$, $k$, let $\mathcal{I}_{j,k}(E)$ denote the collection of the intervals of the form $[l(c2^{dk} + 1)^{-1}2^{-j}, (l+1)(c2^{dk} + 1)^{-1}2^{-j}]$, $l \in \mathbb{Z}$, that intersect $E$. Let $E_{j,k} = \{t_{l}\}_{l}$ be the left endpoints of the intervals belonging to $\mathcal{I}_{j,k}(E)$.
Then
\begin{align}
\sup_{t\in E} |A_{t,j}^kf(y)| % &= \sup_{t_{l}\in E_{j,k}} \sup_{u \in (0, (c2^{dk} +1)^{-1}2^{-j})} |A_{t_{l} + u,j}^kf(y)| \nonumber\\
&\le \sup_{t_{l}\in E_{j,k}} |A_{t_{l},j}^kf(y)| + \int_{0}^{(c2^{dk}+1)^{-1}2^{-j}} \sup_{t_{l}\in E_{j,k}} |\partial_{u} A_{t_{l}+u,j}^kf(y)| du.
\end{align}
Equipping $E_{j,k}$ with the counting measure, we are reduced to proving the following lemma.
\begin{lemma}\label{smallgamma}
For each $j \ge 0$, $k \ge \log(1/\epsilon)$ and $\gamma \in (0,1)$, we have
\[\|A_{t_{l},j}^{k}f \|_{L^{3+2\gamma}({\mathbb{R}}^2 \times E_{j,k})} \lesssim_{\gamma} (c2^{\frac{d\gamma k}{3+2\gamma}} +1)  \max\{j,1\} \|f\|_{L^{\frac{3+2\gamma}{2}}({\mathbb{R}}^2)}. \]
\end{lemma}

The result for $j=0$, $\gamma \in (0,1)$ follows from the estimate (I) in Lemma \ref{basicestimate}. When $j \ge 1$ and $\gamma \in (0,1/2)$, the result in Lemma \ref{smallgamma} can be obtained by interpolating between the inequality
\begin{equation}\label{smallgamma2+}
\|A_{t_{l},j}^{k}f \|_{L^{\infty}({\mathbb{R}}^2 \times E_{j,k})} \lesssim   2^{j} \|f\|_{L^1({\mathbb{R}}^2)}
\end{equation}
and the following estimate from Theorem \ref{smallgamma+} in the appendix,
\begin{equation}\label{smallgamma1+}
\|A_{t_{l},j}^{k}f \|_{L^{2(1+2\gamma), \infty}({\mathbb{R}}^2 \times E_{j,k})} \lesssim (c2^{\frac{d\gamma k}{2(1+2\gamma)}} +1)  2^{-\frac{1-2\gamma}{2(1+2\gamma)}j} \|f\|_{L^2({\mathbb{R}}^2)}.
\end{equation}
More precisely, the interpolation between inequalities (\ref{smallgamma1+}) and (\ref{smallgamma2+}) leads to
\[\|A_{t_{l},j}^{k}f \|_{L^{3+2\gamma}({\mathbb{R}}^2 \times E_{j,k})} \lesssim_{\gamma} (c2^{\frac{d\gamma k}{3+2\gamma}} +1)  \|f\|_{L^{\frac{3+2\gamma}{2}}({\mathbb{R}}^2)}. \]
Moreover, inequality (\ref{smallgamma2+}) can be directly derived from inequality (\ref{Linfty1}). The proof of Theorem \ref{smallgamma+} can be established by using the duality method described in {\cite[Proposition 3.1]{AHRS}}. To avoid disrupting the main thread, we will briefly outline the proof of Theorem  \ref{smallgamma+} in the appendix.

Hence we only consider $j \ge 1$ and $\gamma \in [1/2,1)$ in what follows. Denote
\[F_{j}^{k}(f\otimes f)(y,t_{l})=  \int_{\mathbb{R}^{2}} \int_{\mathbb{R}^{2}}e^{i(y \cdot (\xi +  \zeta)  -\tilde{\Phi}(t_{l},\xi,\delta) - \tilde{\Phi}(t_{l},\zeta,\delta))}a_{j,k}(t_{l},\xi, \zeta) \hat{f}(\xi)\hat{f}(\zeta)d\xi d\zeta,\]
where $a_{j,k}(t_{l}, \xi, \zeta) =a_{j,k,t_{l}}(\xi)a_{j,k,t_{l}}(\zeta) $. It is clear that
\begin{align}
\|A_{t_{l},j}^{k}f \|_{L^{3+2\gamma}({\mathbb{R}}^2 \times E_{j,k})} %&= \biggl( \sum_{t_{l} \in E_{j,k}} \|A_{t_{l},j}^{k}f \|_{L^{3+2\gamma}({\mathbb{R}}^2)}^{3+2\gamma} \biggl)^{\frac{1}{3+2\gamma}} \nonumber\\
%&= \biggl( \sum_{t_{l} \in E_{j,k}} \biggl \| |A_{t_{l},j}^{k}f(y_{1},y_{2} + c2^{dk}t_{l})|^{2} \biggl \|_{L^{\frac{3+2\gamma}{2}}({\mathbb{R}}^2)}^{\frac{3+2\gamma}{2}} \biggl)^{\frac{1}{3+2\gamma}} \nonumber\\
&= 2^{-\frac{j}{2}}\biggl( \sum_{t_{l} \in E_{j,k}} \biggl \| F_{j}^{k}(f\otimes f)(y,t_{l}) \biggl \|_{L^{\frac{3+2\gamma}{2}}({\mathbb{R}}^2)}^{\frac{3+2\gamma}{2}} \biggl)^{\frac{1}{3+2\gamma}}. \nonumber
\end{align}
Hence, we are reduced to proving the following estimate
\begin{equation}\label{blinear3}
\|F_{j}^{k}(f\otimes f)(y,t_{l})\|_{L^{ \frac{3+2\gamma}{2}}({\mathbb{R}}^2 \times E_{j,k})} \lesssim_{\gamma} (c2^{\frac{2d\gamma k}{3+2\gamma}} +1) j  2^{j} \|f\|^2_{L^{\frac{3+2\gamma}{2}}({\mathbb{R}}^2)}.
\end{equation}
By the similar analytic interpolation argument as in {\cite[Page 12-13]{Roos}}, in order to obtain inequality (\ref{blinear3}), we will prove the following two lemmas.

\begin{lemma}\label{bilinear3lemma1}
For each $j \ge 1$, we have
\begin{equation}\label{blinear3L1infty}
\|F_{j}^{k}(f\otimes f)(y,t_{l})\|_{L^{ \infty}({\mathbb{R}}^2 \times E_{j,k})} \lesssim 2^{j} \int_{\mathbb{R}^{2}}\sup_{(w_{2},z_{2})\in \mathbb{R}^{2}}|f(w_{1},w_{2})||f(z_{1},z_{2})|d w_{1}dz_{1}.
\end{equation}
\end{lemma}

\begin{lemma}\label{bilinear3lemma2}
For each $j \ge 1$, there holds
\begin{equation}\label{blinear3L22}
\|F_{j}^{k}(f\otimes f)(y,t_{l})\|_{L^{2}({\mathbb{R}}^2 \times E_{j,k})} \lesssim_{\gamma} (c2^{\frac{d\gamma k}{2}} +1) j  2^{j} \biggl( \int_{\mathbb{R}^{2}} \int_{\mathbb{R}^{2}}|f(w)|^{2}|f(z)|^{2}|w-z|^{-(2\gamma-1)}dwdz \biggl)^{1/2}.
\end{equation}
\end{lemma}

Now we state some details for the proof of inequality (\ref{blinear3}) based on Lemma \ref{bilinear3lemma1} and Lemma \ref{bilinear3lemma2} for the reader's convenience. When $\gamma = \frac{1}{2}$, inequality (\ref{blinear3}) follows from  Lemma \ref{bilinear3lemma2} directly. Notice that $\frac{3+2\gamma}{2}>2$ when $\gamma > \frac{1}{2}$,  then interpolation between (\ref{blinear3L1infty}) and (\ref{blinear3L22}) implies that
\begin{align}
&\|F_{j}^{k}(f\otimes f)(y,t_{l})\|_{L^{\frac{3+2\gamma}{2}}({\mathbb{R}}^2 \times E_{j,k})} \nonumber\\
%&\lesssim_{\gamma} (c2^{\frac{2d\gamma k}{3+2\gamma}} +1) j^{\frac{4}{3+2\gamma}}  2^{j} \biggl\| \bigl\| f(w_{1},w_{2})f(z_{1},z_{2})\bigl\|_{L_{(w_{2},z_{2})}^{\frac{3+2\gamma}{2}}(\mathbb{R}^{2})} |w_{1}-z_{1}|^{-\frac{2(2\gamma-1)}{3+2\gamma}} \biggl\|_{L_{(w_{1},z_{1})}^{\frac{3+2\gamma}{1+2\gamma}}(\mathbb{R}^{2})} \nonumber\\
&\lesssim_{\gamma} (c2^{\frac{2d\gamma k}{3+2\gamma}} +1) j^{\frac{4}{3+2\gamma}}  2^{j}  \biggl\| \|f(w_{1},\cdot ) \|_{L^{\frac{3+2\gamma}{2}}(\mathbb{R})} \big\|  \|f(z_{1},\cdot)\|_{L^{\frac{3+2\gamma}{2}}(\mathbb{R})} |w_{1}-z_{1}|^{-\frac{2(2\gamma-1)}{3+2\gamma}} \bigl\|_{L_{z_{1}}^{\frac{3+2\gamma}{1+2\gamma}}(\mathbb{R})}
\biggl\|_{L_{w_{1}}^{\frac{3+2\gamma}{1+2\gamma}}(\mathbb{R})}. \nonumber
\end{align}
By H\"{o}lder's inequality and the Hardy-Littlewood-Sobolev inequality, we have
\begin{align}
&\biggl\| \|f(w_{1},\cdot ) \|_{L^{\frac{3+2\gamma}{2}}(\mathbb{R})} \bigl\|  \|f(z_{1},\cdot)\|_{L^{\frac{3+2\gamma}{2}}(\mathbb{R})} |w_{1}-z_{1}|^{-\frac{2(2\gamma-1)}{3+2\gamma}} \bigl\|_{L_{z_{1}}^{\frac{3+2\gamma}{1+2\gamma}}(\mathbb{R})}
\biggl\|_{L_{w_{1}}^{\frac{3+2\gamma}{1+2\gamma}}(\mathbb{R})} \nonumber\\
&\lesssim  \|f\|_{L^{\frac{3+2\gamma}{2}}(\mathbb{R}^{2})} \biggl\| \bigl\|  \|f(z_{1},\cdot)\|_{L^{\frac{3+2\gamma}{2}}(\mathbb{R})} |w_{1}-z_{1}|^{-\frac{2(2\gamma-1)}{3+2\gamma}} \bigl\|_{L_{z_{1}}^{\frac{3+2\gamma}{1+2\gamma}}(\mathbb{R})}
\biggl\|_{L_{w_{1}}^{\frac{3+2\gamma}{2\gamma-1}}(\mathbb{R})} \nonumber\\
&\sim_{\gamma} \|f\|_{L^{\frac{3+2\gamma}{2}}(\mathbb{R}^{2})} \biggl\| \int_{\mathbb{R}}   \|f(z_{1},\cdot)\|^{\frac{3+2\gamma}{1+2\gamma}}_{L^{\frac{3+2\gamma}{2}}(\mathbb{R})} |w_{1}-z_{1}|^{-\frac{2(2\gamma-1)}{1+2\gamma}} dz_{1}
\biggl\|^{\frac{1+2\gamma}{3+2\gamma}}_{L_{w_{1}}^{\frac{1+2\gamma}{2\gamma-1}}(\mathbb{R})} \nonumber\\
&\lesssim_{\gamma} \|f\|_{L^{\frac{3+2\gamma}{2}}(\mathbb{R}^{2})} \biggl\|    \|f(w_{1},\cdot)\|^{\frac{3+2\gamma}{1+2\gamma}}_{L^{\frac{3+2\gamma}{2}}(\mathbb{R})}
\biggl\|^{\frac{1+2\gamma}{3+2\gamma}}_{L^{\frac{1+2\gamma}{2}}(\mathbb{R})} \nonumber\\
&\sim_{\gamma}  \|f\|^{2}_{L^{\frac{3+2\gamma}{2}}(\mathbb{R}^{2})}.
\end{align}
Then we arrive at inequality (\ref{blinear3}). Next we will prove Lemma \ref{bilinear3lemma1} and Lemma \ref{bilinear3lemma2}.

\textbf{Proof of Lemma \ref{bilinear3lemma1}.} We fix a collection $\{\kappa_{\nu}\}_{\nu \in \mathbb{N}^{+}}$ of real numbers with $|\kappa_{\nu}| \sim 1$ that satisfies: $(a)$ $|\kappa_{\nu}-\kappa_{\nu^{\prime}}|\geq |\nu - \nu^{\prime}| 2^{-j/2}$, if $\nu\neq\nu'$; $(b)$ if $\xi \in \{\xi \in \mathbb{R}^{2}: \frac{\xi_1}{\xi_2}\sim 1\}$, then there exists a $\kappa_{\nu}$ so that $|\frac{\xi_1}{\xi_2}-\kappa_{\nu}|<2^{-j/2}$. Let
\[K_{t}^{\nu}(y) =2^{-\frac{j}{2}} \int_{\mathbb{R}^{2}}e^{i(y \cdot \xi -\tilde{\Phi}(t,\xi, \delta))} \chi_{\nu}(\frac{\xi_{1}}{\xi_{2}}) a_{j,k,t}(\xi)\hat{f}(\xi)d\xi.\]
Here, $\chi_{\nu}$ is a smooth cutoff function supported in $\{\xi \in \mathbb{R}^{2}: |\frac{\xi_{1}}{\xi_{2}} - \kappa_{\nu}| \le 2^{-\frac{j}{2}}\}$. From the proof of Lemma 2.6 in \cite{LWZ}, for each $y \in \mathbb{R}^{2}$ and $t_{l} \in E_{j,k}$, there holds
\begin{align}
&|F_{j}^{k}(f\otimes f)(y,t_{l})| \nonumber\\
&\le 2^{4j} \int_{\mathbb{R}^{4}} \sum_{\nu}  |K_{t_{l}}^{\nu}(y-z)| \sum_{\nu^{\prime}}  |K_{t_{l}}^{\nu^{\prime}}(y-w)| |f(w)||f(z)|dwdz\nonumber\\
&\le2^{4j} \int_{\mathbb{R}^{2}} \sup_{(w_{2},z_{2}) \in \mathbb{R}^{2}} |f(w_{1},w_{2})||f(z_{1},z_{2})| \nonumber\\
&\quad \quad \times \sum_{\nu}  \int_{\mathbb{R}} |K_{t_{l}}^{\nu}(y_{1}-z_{1},y_{2}-z_{2})| dz_{2} \sum_{\nu^{\prime}} \int_{\mathbb{R}} |K_{t_{l}}^{\nu^{\prime}}(y_{1}-w_{1},y_{2}-w_{2})|dw_{2} dz_{1}dw_{1}, \nonumber
\end{align}
and
\begin{align}
&\sum_{\nu}  \int_{\mathbb{R}} |K_{t_{l}}^{\nu}(y_{1}-z_{1},y_{2}-z_{2})| dz_{2}  \nonumber\\
&\le 2^{-\frac{j}{2}}\sum_{\nu}C _{N} \{(1+ 2^{\frac{j}{2}}|y_{1}-z_{1}- c_{1}(d) t_{l}^{\frac{da_{1}-a_{2}}{d-1}} \kappa_{\nu}^{\frac{1}{d-1}}|   )\}^{-N/2} \nonumber\\
&\quad\times\int_{\mathbb{R}} \bigl\{(1  + 2^{j}|y_{2}-z_{2} +\kappa_{\nu}(y_{1}-z_{1})-c_{2}(d) t_{l}^{\frac{da_{1}-a_{2}}{d-1}} \kappa_{\nu}^{\frac{d}{d-1}}|   )\bigl\}^{-N/2}dz_{2} \nonumber\\
&\lesssim 2^{-\frac{3}{2} j}. \nonumber
\end{align}
Similarly,
\begin{align}
\sum_{\nu'}  \int_{\mathbb{R}} |K_{t_{l}}^{\nu'}(y_{1}-w_{1},y_{2}-w_{2})| dw_{2} \lesssim 2^{-\frac{3}{2} j}. \nonumber
\end{align}
Then we arrive at inequality (\ref{blinear3L1infty}).  $\hfill\square$

\textbf{Proof of Lemma \ref{bilinear3lemma2}.} A change of variables implies that
\[F_{j}^{k}(f\otimes f)(y,t_{l})=  \int_{\mathbb{R}^{2}} \int_{\mathbb{R}^{2}}e^{i(y \cdot \xi - \tilde{\Phi}(t_{l},\xi -\zeta, \delta) - \tilde{\Phi}(t_{l},\zeta, \delta)) }a_{j,k}(t_{l},\xi-\zeta, \zeta) \hat{f}(\xi-\zeta)\hat{f}(\zeta) d\zeta d\xi,\]
then by Plancherel's theorem,
\begin{align}
\|F_{j}^{k}(f\otimes f)(y,t_{l})\|_{L^{2}({\mathbb{R}}^2 \times E_{j,k})} = \biggl\| \int_{\mathbb{R}^{2}}e^{-i(\tilde{\Phi}(t_{l},\xi -\zeta, \delta)+ \tilde{\Phi}(t_{l},\zeta, \delta)) }a_{j,k}(t_{l},\xi-\zeta, \zeta) \hat{f}(\xi-\zeta)\hat{f}(\zeta) d\zeta \biggl\|_{L^{2}({\mathbb{R}}^2 \times E_{j,k})}. \nonumber
\end{align}
We note that $\zeta_{1} \sim 2^{j}$, $\zeta_{2} \sim 2^{j}$, $\xi_{1} - \zeta_{1} \sim 2^{j}$, $\xi_{2}-\zeta_{2} \sim 2^{j}$, which imply that $\xi_{1} \sim 2^{j}$, $\xi_{2} \sim 2^{j}$ and $\xi_{1}/\xi_{2} \sim 1$. Hence $|\frac{\xi_{1}}{\xi_{2}}- \frac{\zeta_{1}}{\zeta_{2}}| \lesssim 1$. Choose a non-negative bump function $b$ such that supp $b \subset B(0,C_{d,k,m})$ and denote
\[T_{j}^{k}(\mathcal{F}, b)(\xi,t_{l})=  \int_{\mathbb{R}^{2}}e^{-i(  \tilde{\Phi}(t_{l},\xi -\zeta, \delta) + \tilde{\Phi}(t_{l},\zeta, \delta)) }b \biggl(\frac{\xi_{1}}{\xi_{2}}-\frac{\zeta_{1}}{\zeta_{2}} \biggl)a_{j,k}(t_{l},\xi-\zeta, \zeta) \hat{\mathcal{F}}(\xi-\zeta,\zeta) d\zeta, \]
where $\hat{\mathcal{F}}(\xi-\zeta, \zeta) =\hat{f}(\xi - \zeta)\hat{f}(\zeta)$. We may assume that $C_{d,k,m}=1$ for convenience. Then we decompose $T_{j}^{k}(\mathcal{F},b)(\xi,t_{l}) = \sum_{h=1}^{j/2}T_{j}^{k}(\mathcal{F},b_{h})(\xi,t_{l})$, where $b_{h}$'s are even functions and supp $b_{h} \subset \{ \tau \in \mathbb{R}: |\tau| \sim 2^{-h} \}$ for $1 \le h < \frac{j}{2}$ and supp $b_{\frac{j}{2}} \subset \{ \tau \in \mathbb{R}: |\tau| \lesssim 2^{-j/2} \}$. For fixed $h$, we further decompose $\mathcal{F} = \sum_{n=0}^{\infty} \mathcal{F}_{-h+n}$ with supp $\mathcal{F}_{-h} \subset \{(w,z) \in \mathbb{R}^{2} \times \mathbb{R}^{2}: |z-w| \lesssim 2^{-h}\}$ and supp $\mathcal{F}_{-h+n} \subset \{(w,z) \in \mathbb{R}^{2} \times \mathbb{R}^{2}: |z-w| \sim 2^{-h+n}\}$, $n \ge 1$. By the similar argument as in {\cite[Page 15-17]{Roos}}, Lemma \ref{bilinear3lemma2} can be proved if the following propositions are established.

\begin{proposition}\label{keylemma1}
For each $j \geq1$  and $1 \le h \le \frac{j}{2}$, we have
\begin{equation}
\|T_{j}^{k}(\mathcal{F}_{-h}, b_{h})\|_{L^{2}(\mathbb{R}^{2} \times E_{j,k})} \lesssim 2^{j-\frac{h}{2}} (c2^{\frac{\gamma dk}{2}} +1) 2^{\gamma h} \|\mathcal{F}_{-h}\|_{L^{2}(\mathbb{R}^{2} \times \mathbb{R}^{2})}.
\end{equation}
\end{proposition}

\begin{proposition}\label{keylemma2}
For each $j \geq 1$, $n \ge 1$ and $1 \le h \le \frac{j}{2}$, there holds
\begin{equation}
\|T_{j}^{k}(\mathcal{F}_{-h+n}, b_{h})\|_{L^{2}(\mathbb{R}^{2} \times E_{j,k})} \lesssim 2^{\frac{3}{2}h-n} (c2^{\frac{\gamma dk}{2}} +1) 2^{\frac{\gamma j}{2}} \|\mathcal{F}_{-h+n}\|_{L^{2}(\mathbb{R}^{2} \times \mathbb{R}^{2})}.
\end{equation}
\end{proposition}

More precisely, if the above two propositions hold true, then
\begin{align}
&\|T_{j}^{k}(\mathcal{F}, b)(y,t_{l})\|_{L^{2}({\mathbb{R}}^2 \times E_{j,k})} \nonumber\\
&\le  \sum_{h=1}^{j/2}
\|T_{j}^{k}(\mathcal{F}_{-h}, b_{h})(y,t_{l})\|_{L^{2}({\mathbb{R}}^2 \times E_{j,k})} +  \sum_{h=1}^{j/2} \sum_{n=1}^{\infty} \|T_{j}^{k}(\mathcal{F}_{-h+n}, b_{h})(y,t_{l})\|_{L^{2}({\mathbb{R}}^2 \times E_{j,k})} \nonumber\\
&\le  \sum_{h=1}^{j/2} 2^{j}(c2^{\frac{\gamma dk}{2}} +1)\biggl( \int_{\mathbb{R}^{2}} \int_{\mathbb{R}^{2}}|f(w)|^{2}|f(z)|^{2}|w-z|^{-(2\gamma-1)}dwdz \biggl)^{1/2} \nonumber\\
&\quad \quad + \sum_{h=1}^{j/2} \sum_{n=1}^{\infty} 2^{\frac{3}{2}h-n} (c2^{\frac{\gamma dk}{2}} +1) 2^{\frac{\gamma j}{2}} 2^{(-h+n)(\gamma - \frac{1}{2})}  \biggl( \int_{\mathbb{R}^{2}} \int_{\mathbb{R}^{2}}|f(w)|^{2}|f(z)|^{2}|w-z|^{-(2\gamma-1)}dwdz \biggl)^{1/2} \nonumber\\
 % &\le  j 2^{j}\biggl( \int_{\mathbb{R}^{2}} \int_{\mathbb{R}^{2}}|f(w)|^{2}|f(z)|^{2}|w-z|^{-(2\gamma-1)}dwdz \biggl)^{1/2} \nonumber\\
 % &\quad \quad + \sum_{h=1}^{\frac{j-mk}{2}}  2^{(2-\gamma)h} 2^{\frac{\gamma j}{2}}2^{-\frac{\gamma mk}{2}} \biggl( \int_{\mathbb{R}^{2}} \int_{\mathbb{R}^{2}}|f(w)|^{2}|f(z)|^{2}|w-z|^{-(2\gamma-1)}dwdz \biggl)^{1/2} \nonumber\\
&\lesssim  j 2^{j}(c2^{\frac{\gamma dk}{2}}+1)\biggl( \int_{\mathbb{R}^{2}} \int_{\mathbb{R}^{2}}|f(w)|^{2}|f(z)|^{2}|w-z|^{-(2\gamma-1)}dwdz \biggl)^{1/2}.
\end{align}

The original idea for proving   the above two propositions is derived from {\cite[Proposition 4.2 and Proposition 4.3]{Roos}}. Hence we first write the details which are different in our case, and then include some similar details for completeness.

\textbf{Proof of Proposition \ref{keylemma1}:} The case when $h = \frac{j}{2}$ can be easily obtained by Cauchy-Schwarz inequality and Plancherel's theorem. In order to prove the case when $1 \le h < \frac{j}{2}$, we define
\[K(t_{l},t_{l^{\prime}}, \xi) =  \int_{\mathbb{R}^{2}}e^{i \Phi_{\xi,\delta}(t_{l},t_{l^{\prime}},\zeta)}b^{2}_{h} \biggl(\frac{\xi_{1}}{\xi_{2}}-\frac{\zeta_{1}}{\zeta_{2}} \biggl)a_{j,k,t_{l},t_{l^{\prime}}}(\xi, \zeta) d\zeta,\]
where  $a_{j,k,t_{l},t_{l^{\prime}}}(\xi, \zeta)= a_{j,k}(t_{l},\xi-\zeta, \zeta) a_{j,k}(t_{l^{\prime}},\xi-\zeta, \zeta) $, and
\[\Phi_{\xi,\delta}(t_{l},t_{l^{\prime}},\zeta) = \tilde{\Phi}(t_{l^{\prime}},\xi-\zeta, \delta)-\tilde{\Phi}(t_{l},\xi-\zeta, \delta)+\tilde{\Phi}(t_{l^{\prime}},\zeta,\delta)-\tilde{\Phi}(t_{l},\zeta,\delta). \]
By the similar argument as in the proof of {\cite[Proposition 4.2]{Roos}}, the case when $1 \le h < \frac{j}{2}$
 can be established provided that we have proved the following estimate
\begin{equation}\label{orthognality1}
|K(t_{l},t_{l^{\prime}}, \xi)| \le 2^{2j-h} \frac{C_{N}}{(1+2^{j-2h}|t_{l}-t_{l^{\prime}}|)^{N}}, \quad \quad N \gg 1
\end{equation}
for each $\xi$ and  $t_{l},t_{l^{\prime}} \in E_{j,k}$. %Next we will prove inequality (\ref{orthognality1}).
Changing variables $\zeta= B\eta$, where $B=\left(
\begin{array} {lcr}
2^{-h}\xi_{2} & \xi_{1}  \\
0 & \xi_{2}
\end{array}
\right)$, then we get
\[K(t_{l},t_{l^{\prime}}, \xi) = \xi_{2}^{2} 2^{-h}  \int_{\mathbb{R}^{2}}e^{i \Phi_{h, \xi,\delta}(t_{l},t_{l^{\prime}},\eta)}\tilde{b}^{2} \biggl(\frac{\eta_{1}}{\eta_{2}} \biggl)a_{j,k,h,t_{l},t_{l^{\prime}},\xi}(\eta) d\eta,\]
where $\tilde{b}$ is a bump function supported in $\{r \in \mathbb{R}: |r| \sim 1\}$,
\[\Phi_{h,\xi,\delta}(t_{l},t_{l^{\prime}},\eta)= \Phi_{\xi, \delta}(t_{l},t_{l^{\prime}}, B\eta),\]
and
\[a_{j,k,h,t_{l},t_{l^{\prime}},\xi}(\eta)= a_{j,k,t_{l},t_{l^{\prime}}}(\xi, B \eta).\]
By Taylor's expansion, $\Phi_{h, \xi, \delta}(t_{l},t_{l^{\prime}},\eta)$ can be written as
\begin{align}
% &= (d-1)\xi_{2} (1-\eta_{2})  \biggl[-\frac{\xi_{1}}{d\xi_{2}} + 2^{-h}\frac{\eta_{1}}{d(1-\eta_{2})} \biggl]^{\frac{d}{d-1}}  + (d-1) \xi_{2} \eta_{2}  \biggl(-\frac{\xi_{1}}{d\xi_{2}} -2^{-h}\frac{\eta_{1}}{d\eta_{2}} \biggl)^{\frac{d}{d-1}} \nonumber\\
 %&= (d-1)\xi_{2} (1-\eta_{2})  \biggl(-\frac{\xi_{1}}{d\xi_{2}} \biggl)^{\frac{d}{d-1}} + 2^{-h}\xi_{2} \eta_{1} \biggl(-\frac{\xi_{1}}{d\xi_{2}} \biggl)^{\frac{1}{d-1}} + \frac{2^{-2h-1}\xi_{2}}{d(d-1)}\biggl(-\frac{\xi_{1}}{d\xi_{2}} \biggl)^{\frac{2-d}{d-1}} \frac{\eta_{1}^{2}}{(1-\eta_{2})} \nonumber\\
 % &\quad \quad + (d-1) \xi_{2} \eta_{2}  \biggl(-\frac{\xi_{1}}{d\xi_{2}} \biggl)^{\frac{d}{d-1}} - 2^{-h}\xi_{2} \eta_{1} \biggl(-\frac{\xi_{1}}{d\xi_{2}} \biggl)^{\frac{1}{d-1}} + \frac{2^{-2h-1}\xi_{2}}{d(d-1)}\biggl(-\frac{\xi_{1}}{d\xi_{2}} \biggl)^{\frac{2-d}{d-1}} \frac{\eta_{1}^{2}}{\eta_{2}} + 2^{-3h} \xi_{2} R(\xi, \eta) \nonumber\\
  &\biggl(t_{l}^{\frac{da_{1}-a_{2}}{d-1}}- t_{l^{\prime}}^{\frac{da_{1}-a_{2}}{d-1}}\biggl) \biggl[(d-1) \xi_{2}  \biggl(-\frac{\xi_{1}}{d\xi_{2}} \biggl)^{\frac{d}{d-1}}  + \frac{2^{-2h-1}\xi_{2}}{d(d-1)}\biggl(-\frac{\xi_{1}}{d\xi_{2}} \biggl)^{\frac{2-d}{d-1}} \biggl( \frac{\eta_{1}^{2}}{\eta_{2}} +  \frac{\eta_{1}^{2}}{1-\eta_{2}} \biggl)\biggl] \nonumber\\
  & -(t^{\frac{(a_{1}-a_{2})(m+d)}{d-1}+a_{2}}_{l}-t^{\frac{(a_{1}-a_{2})(m+d)}{d-1}+a_{2}}_{l^{\prime}})\frac{\delta^{m} \phi^{m}(0)}{m!} \xi_{2} \biggl[ \biggl( \frac{\xi_{1}}{-d \xi_{2}}\biggl)^{\frac{d+m}{d-1}} + 2^{-2h-1}\biggl( \frac{\xi_{1}}{-d \xi_{2}}\biggl)^{\frac{m-d+2}{d-1}} \nonumber\\
  &\times \frac{(m+1)(m+d)}{d^{2}(d-1)^{2}}   \biggl(\frac{\eta_{1}^{2}}{1-\eta_{2}}+ \frac{\eta_{1}^{2}}{\eta_{2}}  \biggl) \biggl]   + 2^{-3h} \xi_{2} R_{t_{l},t_{l^{\prime}}}(\xi, \eta,\delta),
\end{align}
where we put the remainder terms into $R_{t_{l},t_{l^{\prime}}}(\xi, \eta,\delta)$.  Notice that $\xi_{2} \sim 2^{j}$ and $\xi_{2} - \zeta_{2} \sim 2^{j}$ imply $1-\eta_{2} \sim 1$, $\xi_{2} \sim 2^{j}$ and $\zeta_{2} \sim 2^{j}$ imply $\eta_{2} \sim 1$. Since $\eta_{1}/\eta_{2} \sim 1$, then we have $\eta_{1}/(1-\eta_{2}) \sim 1$. Now it is clear that
\[|\nabla_{\eta} \Phi_{h, \xi,\delta}(t_{l},t_{l^{\prime}},\eta)| \sim 2^{j-2h}\biggl|t_{l^{\prime}}^{\frac{da_{2}-a_{1}}{d-1}}-t_{l}^{\frac{da_{2}-a_{1}}{d-1}}\biggl| \sim 2^{j-2h} |t_{l}-t_{l^{\prime}}|,\]
and
\[|D_{\eta}^{\alpha} \Phi_{h, \xi,\delta}(t_{l},t_{l^{\prime}},\eta)| \lesssim 2^{j-2h} |t_{l^{\prime}} -t_{l}   |, \quad \quad |\alpha| \ge 2.\]
Since $a_{j,k,h,t_{l},t_{l^{\prime}}, \xi}$ is a symbol of order zero, we arrive at inequality (\ref{orthognality1}) by integration by parts.

For completeness, we prove Proposition \ref{keylemma1} based on the kernel estimate (\ref{orthognality1}). We decompose the set $E_{j,k}$ into $E_{j,k}=\cup_{s=1}^{M}E_{j,k,s}$, where each $E_{j,k,s}$ satisfies $E_{j,k,s}-E_{j,k,s} \subset 2^{-j+2h}\mathbb{Z}$. According to inequality (3.4) in \cite{AHRS}, we have $M \lesssim (c2^{\gamma dk}+c)2^{2\gamma h}$. For fixed $E_{j,k,s}$, by duality,
\begin{align}
&\|T_{j}^{k}(\mathcal{F}_{-h}, b_{h})\|_{L^{2}(\mathbb{R}^{2} \times E_{j,k,s})} \nonumber\\
 &=\sup_{G: \|G\|_{L^{2}(\mathbb{R}^{2} \times E_{j,k,s})} \le 1} \biggl|\int_{\mathbb{R}^{2}}  \int_{\mathbb{R}^{2}} \hat{\mathcal{F}}(\xi-\zeta,\zeta) \sum_{t_{l} \in E_{j,k,s}}e^{-i\{  \tilde{\Phi}(t_{l},\xi -\zeta, \delta) + \tilde{\Phi}(t_{l},\zeta, \delta)) }\nonumber\\
 &\quad \times b_{h} \biggl(\frac{\xi_{1}}{\xi_{2}}-\frac{\zeta_{1}}{\zeta_{2}} \biggl) a_{j,k}(t_{l},\xi-\zeta, \zeta)G(\xi,t_{l}) d\xi d\zeta \biggl| \nonumber\\
 &\le \|\mathcal{F}\|_{L^{2}(\mathbb{R}^{4})}\sup_{G: \|G\|_{L^{2}(\mathbb{R}^{2} \times E_{j,k,s}) } \le 1} \biggl(\sum_{t_{l} \in E_{j,k,s}} \sum_{t_{l^{\prime}} \in E_{j,k,s}} \int_{\mathbb{R}^{2}} K(t_{l},t_{l^{\prime}},\xi)G(\xi,t_{l}) \bar{G}(\xi,t_{l^{\prime}}) d\xi \biggl)^{\frac{1}{2}}. \nonumber
\end{align}
If (\ref{orthognality1}) remains valid, then we have
\begin{align}
&\sum_{t_{l} \in E_{j,k,s}} \sum_{t_{l^{\prime}} \in E_{j,k,s}} \int_{\mathbb{R}^{2}} K(t_{l},t_{l^{\prime}},\xi)G(\xi,t_{l}) \bar{G}(\xi,t_{l^{\prime}}) d\xi  \nonumber\\
&\le C_{N} 2^{2j-h} \sum_{t_{l} \in E_{j,k,s}} \sum_{t_{l^{\prime}} \in E_{j,k,s}} \int_{\mathbb{R}^{2}} \frac{ |G(\xi,t_{l}) \bar{G}(\xi,t_{l^{\prime}})| }{(1+2^{j-2h}|t_{l}-t_{l^{\prime}}|)^{N}}d\xi \nonumber\\
&\le C_{N} 2^{2j-h} \sum_{t \in E_{j,k,s}-E_{j,k,s}} \sum_{t_{l^{\prime}} \in E_{j,k,s}} \int_{\mathbb{R}^{2}} \frac{ |G(\xi,t_{l^{\prime}}+t) \bar{G}(\xi,t_{l^{\prime}})| }{(1+2^{j-2h}|t|)^{N}}d\xi,
\end{align}
where we have used Cauchy-Schwarz inequality on the sum over $t_{l'}$, and the fact that $E_{j,k,s}-E_{j,k,s} \subset 2^{-j+2h}\mathbb{Z}$ to sum over $t$. Then Proposition \ref{keylemma1}  follows.
 $\hfill\square$

\textbf{Proof of Proposition \ref{keylemma2}.}  By the proof of {\cite[Proposition 4.3]{Roos}}, it suffices to show that
\begin{equation}\label{offdianol}
\biggl\| \int_{\mathbb{R}^{2}}e^{i w \cdot \zeta} a_{t_{l}, j,k,h,\xi}(\zeta,w) \psi(2^{h-n}w) \hat{\mathcal{F}}^{y}_{-h+n}(\xi, \xi -w)dw \biggl\|_{L_{\zeta}^{2}(\mathbb{R}^{2})} \lesssim 2^{2h-j-n} \|\hat{\mathcal{F}}_{-h+n}(\xi, \xi - \zeta)\|_{L_{\zeta}^{2}(\mathbb{R}^{2})},
\end{equation}
where $\psi$ is a bump function supported in $\{r \in \mathbb{R}: |r| \sim 1\}$, $\hat{\mathcal{F}}_{-h+n}^{y}(\xi,\xi-w)$ means the Fourier transform of $\mathcal{F}_{-h+n}(y,y-w)$ with respect to $y$, and
\[a_{t_{l},j,k,h, \xi}(\zeta,w) = i \frac{\langle \nabla_{\zeta} \tilde{a}_{t_{l},j,k,h, \xi}, \nabla_{\zeta} \Phi_{\xi,w,t_{l},\delta} \rangle}{|\nabla_{\zeta} \Phi_{\xi,w,t_{l},\delta}|^{2}} - 2i \tilde{a}_{t_{l},j,k,h, \xi} \frac{\langle H_{\zeta}(\Phi_{\xi,w,t_{l},\delta}) \nabla_{\zeta} \Phi_{\xi,w,t_{l},\delta}, \nabla_{\zeta} \Phi_{\xi,w,t_{l},\delta} \rangle}{|\nabla_{\zeta} \Phi_{\xi,w,t_{l},\delta}|^{4}}. \]
Here, $\tilde{a}_{t_{l},j,k,h,\xi}( \zeta)= b_{h}(\frac{\xi_{1}}{\xi_{2}}-\frac{\zeta_{1}}{\zeta_{2}})a_{j,k}(t_{l}, \xi-\zeta, \zeta)$, $\Phi_{\xi,w,t_{l},\delta}(\zeta)= w \cdot \zeta - \tilde{\Phi}(t_{l},\xi-\zeta,\delta)- \tilde{\Phi}(t_{l},\zeta,\delta)$, and $H_{\zeta}$ means the Heissian matrix with respect to $\zeta$. According to \cite{CV}, inequality (\ref{offdianol}) follows if we can show that for each multi-index $\alpha$ and $\beta$, there holds
\begin{equation}\label{quasidifferntialopertor}
|D_{\zeta}^{\beta}D_{w}^{\alpha} [a_{t_{l},j,k,h, \xi}(\zeta,w)] | \le 2^{2h-j-n} (1+|\zeta|)^{-\frac{1}{2}(|\beta|-|\alpha|)}.
\end{equation}
Indeed, %since
%\begin{align}
%&\nabla_{\zeta}\Phi_{\xi,w,t_{l},\delta}(\zeta) \nonumber\\
%&=\biggl(w_{1} + \biggl[-\frac{\xi_{1} - \zeta_{1}}{d(\xi_{2}-\zeta_{2})}\biggl]^{\frac{1}{d-1}} - \biggl(-\frac{
%\zeta_{1}}{d \zeta_{2}}\biggl)^{\frac{1}{d-1}} + o(2^{-mk-h}),w_{2} +  \biggl[-\frac{\xi_{1} - %\zeta_{1}}{d(\xi_{2}-\zeta_{2})}\biggl]^{\frac{d}{d-1}} - \biggl(-\frac{
%\zeta_{1}}{d \zeta_{2}}\biggl)^{\frac{d}{d-1}} \biggl)+ o(2^{-mk-h}).\nonumber
%\end{align}
changing variables $\zeta = B\eta$ with  $B=\left(
\begin{array} {lcr}
2^{-h}\xi_{2} & \xi_{1}  \\
0 & \xi_{2}
\end{array}
\right)$,
we obtain
\[\biggl|\frac{\xi_{1}-\zeta_{1}}{\xi_{2}-\zeta_{2}} - \frac{\zeta_{1}}{\zeta_{2}} \biggl| = 2^{-h} \biggl|\frac{\eta_{1}}{1-\eta_{2}} + \frac{\eta_{1}}{\eta_{2}}\biggl| \sim 2^{-h}.\]
Then the mean value theorem and the fact that $|w| \sim 2^{-h+n}$ imply that
\[|\nabla_{\zeta}\Phi_{\xi,w,t_{l},\delta}(\zeta)| \sim 2^{-h+n}.\]
Moreover, it is easy to observe that $|\nabla_{\zeta} \tilde{a}_{t_{l},j,k,h,\xi}(\zeta)| \lesssim 2^{h-j}$ and $|D^{\alpha}_{\zeta}\Phi_{\xi,w,t_{l},\delta}(\zeta)| \lesssim 2^{-j}$ for each $\alpha$ such that $|\alpha| = 2$.
Therefore, we have
\[| a_{t_{l},j,k,h, \xi}(\zeta,w) | \le 2^{2h-j-n},\]
and
\[|D_{\zeta}^{\beta}D_{w}^{\alpha} [a_{t_{l},j,k,h, \xi}(\zeta,w)] | \le 2^{2h-j-n} 2^{(h-n)|\alpha|} 2^{(h-j)|\beta|},  \]
for each multi-index $\alpha$ and $\beta$. Then inequality (\ref{quasidifferntialopertor}) follows from $h \le \frac{j}{2}$ and $|\zeta| \sim 2^{j}$.

We briefly explain how to derive Proposition \ref{keylemma2} from inequality (\ref{offdianol}), we write
\begin{align}
&T_{j}^{k}(\mathcal{F}_{-h+n}, b_{h})(\xi,t_{l}) \nonumber\\
&= \int_{\mathbb{R}^{4}}e^{-i(  \tilde{\Phi}(t_{l},\xi -\zeta, \delta) + \tilde{\Phi}(t_{l},\zeta, \delta)) }b_{h} \biggl(\frac{\xi_{1}}{\xi_{2}}-\frac{\zeta_{1}}{\zeta_{2}} \biggl)a_{j,k}(t_{l},\xi-\zeta, \zeta) \hat{\mathcal{F}}_{-h+n}(\xi-\zeta,\zeta) d\zeta \nonumber\\
&=\int_{\mathbb{R}^{4}} e^{-iy\cdot \xi} \psi(2^{h-n}(y-z)) \mathcal{F}_{-h+n}(y,z) \int_{\mathbb{R}^{2}}e^{i(y-z) \cdot \zeta-i( \tilde{\Phi}(t_{l},\xi -\zeta, \delta) + \tilde{\Phi}(t_{l},\zeta, \delta)) }b_{h} \biggl(\frac{\xi_{1}}{\xi_{2}}-\frac{\zeta_{1}}{\zeta_{2}} \biggl) \nonumber\\
&\quad \times a_{j,k}(t_{l},\xi-\zeta, \zeta) d\zeta dy dz.
\end{align}
Next we perform integration by parts with respect to $\zeta$,   then change the order of integration over $\zeta$ and that over $y$. Combining with a change of variable $w = y-z$, we get
\begin{align}
&|T_{j}^{k}(\mathcal{F}_{-h+n}, b_{h})(\xi,t_{l})| \nonumber\\
&=\biggl|\int_{\mathbb{R}^{2}}  \int_{\mathbb{R}^{2}}e^{iw \cdot \zeta-i(  \tilde{\Phi}(t_{l},\xi -\zeta, \delta) + \tilde{\Phi}(t_{l},\zeta, \delta))\psi(2^{h-n}w)  }a_{t_{l},j,k,h, \xi}(\zeta,w)\psi(2^{h-n}w)  \hat{\mathcal{F}}_{-h+n}(\xi,\xi-w)  dw d\zeta \biggl| \nonumber\\
&\le 2^{j-\frac{h}{2}} \biggl\| \int_{\mathbb{R}^{2}}e^{i w \cdot \zeta} a_{t_{l}, j,k,h,\xi}(\zeta,w) \psi(2^{h-n}w) \hat{\mathcal{F}}^{y}_{-h+n}(\xi, \xi -w)dw \biggl\|_{L_{\zeta}^{2}(\mathbb{R}^{2})} \nonumber\\
&\le  2^{\frac{3h}{2}-n} \|\hat{\mathcal{F}}_{-h+n}(\xi, \xi - \zeta)\|_{L_{\zeta}^{2}(\mathbb{R}^{2})},
\end{align}
where we have applied Cauchy-Schwarz inequality  and  (\ref{offdianol}) to obtain the last two inequalities. Then we arrive at Proposition \ref{keylemma2} since $Card(E_{j,k}) \lesssim (c2^{d\gamma k}+1) 2^{\gamma j}$.
 $\hfill\square$

\section{A counterexample}\label{sharpresultsection}
\begin{theorem}\label{sharpresult}
Let
\begin{equation}
A_{t}f(y):= \int_{0}^{1} f(y_1-t^{a_{1}}x,y_2-t^{a_{2}}(x^d +c)) dx, \quad da_{1} \neq a_{2}.
\end{equation}
There exist some sets $E \subset [1,2]$ such that the $L^{p}\rightarrow L^{q}$ estimate of the maximal operator $\sup_{t \in E}|A_{t}|$ fails for $(\frac{1}{p}, \frac{1}{q}) \notin \overline{\widetilde{\Delta}}$ when $c \neq 0$, and for $(\frac{1}{p}, \frac{1}{q}) \notin \overline{\Delta}$ when   $c=0$.
\end{theorem}

\textbf{Proof.} We first consider the case when $c \neq 0$. Without loss of generality, we may assume that $c=1$.
According to {\cite[Theorem 2.7]{LWZ}}, for fixed $t \in [1,2]$, the $L^{p}(\mathbb{R}^{2}) \rightarrow L^{q}(\mathbb{R}^{2})$ estimate of $A_{t}f$ cannot hold either $\frac{1}{q} > \frac{1}{p}$ or $\frac{2}{q} <\frac{1}{p}$.
The similar proof as in {\cite[Section 4.2]{AHRS}} implies that for any subset $E \subset [1,2]$  with $dim_{M}(E) = \beta$, the $L^{p}(\mathbb{R}^{2}) \rightarrow L^{q}(\mathbb{R}^{2})$  estimate of $\sup_{t\in E} |A_{t}f|$ may hold true only if $1 + \frac{1-\beta}{q} \ge \frac{2}{p}$. Hence, we only need to show that there exists a subset $E \subset [1,2]$  with $dim_{M}(E) = \beta$, $dim_{A}(E) = \gamma$, such that the $L^{p}(\mathbb{R}^{2}) \rightarrow L^{q}(\mathbb{R}^{2})$  estimate of $\sup_{t\in E} |A_{t}f|$ fails if
\begin{equation}\label{N1}
\frac{1}{q} < \frac{d+1}{d-d\beta+1}\biggl(\frac{1}{p} -\frac{1}{d+1} \biggl)
\end{equation}
or
\begin{equation}\label{N2}
 \frac{1}{q} < \frac{d \gamma + \beta}{2(d \gamma + \beta)-\beta(d+1+d \gamma) }\biggl( \frac{1}{p} -\frac{\beta}{d\gamma + \beta} \biggl).
 \end{equation}
Inspired by {\cite[Theorem 2]{AHRS}}, we will utilize  the $(\beta, \gamma)$-Assouad  regular set $E$ constructed in {\cite[Section 5]{AHRS}} to prove the above conclusions.

For this goal, we denote the Assouad spectrum $dim_{A,\theta}(E) = \tilde{\gamma}_{\theta}$ for $\theta \in [0,1)$. The definition and some properties of the Assouad spectrum can be found in \cite{AHRS}.
We will prove that  if
\begin{equation}\label{Assouadcondition}
dim_{A,\theta}(E) = \frac{\beta}{1-\theta},
\end{equation}
then the $L^{p}(\mathbb{R}^{2}) \rightarrow L^{q}(\mathbb{R}^{2})$  estimate of $\sup_{t\in E} |A_{t}f|$ fails if
\begin{equation}\label{N3}
\frac{1}{d} \frac{\beta}{\tilde{\gamma}_{\theta}} + \biggl( 2 - \beta- \frac{d-1}{d} \frac{\beta}{\tilde{\gamma}_{\theta}} \biggl) \frac{1}{q} < \biggl( 1 + \frac{1}{d} \frac{\beta}{\tilde{\gamma}_{\theta}} \biggl) \frac{1}{p}.
\end{equation}
According to {\cite[Lemma 5.1]{AHRS}}, the equality (\ref{Assouadcondition}) holds true when $\theta = 1- \frac{\beta}{\gamma}$ and $\tilde{\gamma}_{\theta} = \gamma$, then the inequality (\ref{N3}) is equivalent to  (\ref{N2}). Moreover, if $\theta = 0$, then  equality (\ref{Assouadcondition}) is trivial and $dim_{A,0}(E) = \beta$, inequality (\ref{N3})  coincides with  (\ref{N1}).

\begin{center}
\includegraphics[height=7cm]{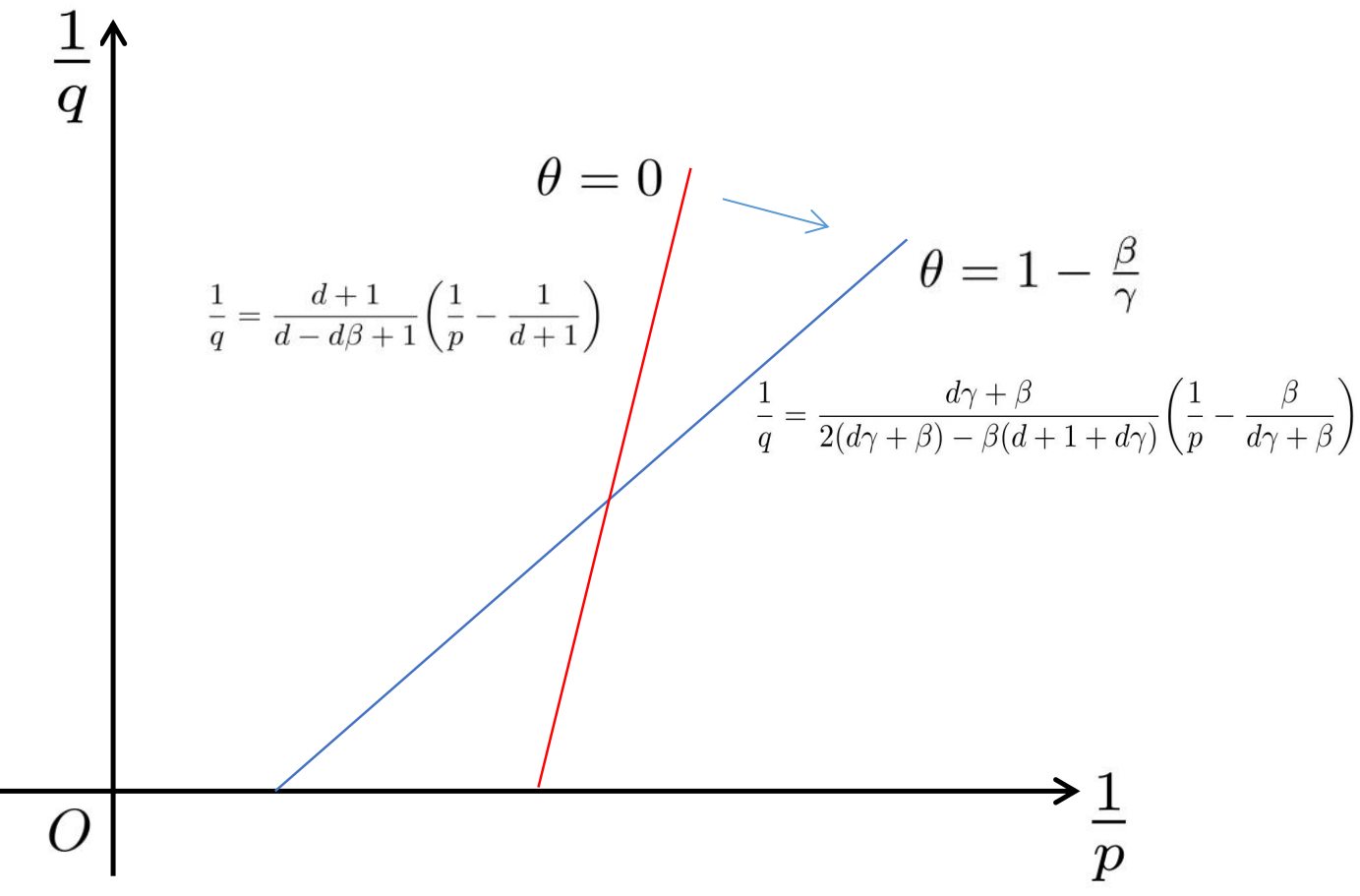}
\end{center}
\begin{center}
Figure 3. If $\theta = 0$, the inequality (\ref{N3})  coincides with  (\ref{N1}). When $\theta = 1- \frac{\beta}{\gamma}$, then the inequality (\ref{N3}) is equivalent to  (\ref{N2}).
\end{center}

Next, we will prove that the $L^{p}(\mathbb{R}^{2}) \rightarrow L^{q}(\mathbb{R}^{2})$  estimate of $\sup_{t\in E} |A_{t}f|$ fails for $(\frac{1}{p}, \frac{1}{q})$ satisfying inequality (\ref{N3}) under the assumption (\ref{Assouadcondition}). We choose an interval $I \subset [1,2]$ with $|I|=\delta^{\theta}$ for some small $\delta  \in (0,1)$ and denote its left endpoint by $r$. Let $\sigma = \delta^{\frac{1-\theta}{d}}$. For each $(r^{a_{1}}x, r^{a_{2}}x^{d}+r^{a_{2}})$, $x \in (0, \sigma)$, denote
\begin{align}
T_{(r^{a_{1}}x,r^{a_{2}}x^{d}+r^{a_{2}})}=&\{(y_{1},y_{2}): |(y_{1} + r^{a_{1}} x,y_{2} + r^{a_{2}} x^{d} + r^{a_{2}} ) \cdot e_{1}| \lesssim   \delta,\nonumber \\
& \hspace{0.2cm} \textmd{and} \hspace{0.2cm} |(y_{1}+r^{a_{1}} x,y_{2}+ r^{a_{2}} x^{d}+r^{a_{2}} ) \cdot e_{2}| \lesssim  \delta \sigma^{-(d-1)} \}, \nonumber
\end{align}
where
\[e_{1} =\frac{(-  r^{a_{2}}dx^{d-1},r^{a_{1}})}{ |(-  r^{a_{2}}dx^{d-1},r^{a_{1}})|},  \quad \quad e_{2}=\frac{(-r^{a_{1}},-  r^{a_{2}}dx^{d-1})}{ |(-r^{a_{1}},- r^{a_{2}} dx^{d-1})|}.\]
Put
\[U= \bigcup_{x \in (0,\sigma)}T_{(r^{a_{1}}x,r^{a_{2}}x^{d}+r^{a_{2}})}.\]
Then it can be observed that $|U| \lesssim \delta \sigma$, where the implied constant only depends on $d$ and $a_{1}$, $a_{2}$. Put $f = \chi_{U}$. Here $\chi$ is the characteristic function over the set $U$. Then
\begin{equation}\label{upperboundlp}
\|f\|_{L^{p}(\mathbb{R}^{2})} \lesssim (\delta \sigma)^{\frac{1}{p}}.
\end{equation}

For each $t \in E \cap I$, we denote
\[D_{t}= \{(z_{1},z_{2} + t^{a_{2}}  - r^{a_{2}} ): |z_{1}| \le  \delta \sigma^{-(d-1)}, |z_{2}| \le  \delta \}.\]
Then for each $(y_{1},y_{2}) \in D_{t}$, we claim that
\begin{equation}\label{lowerboundintegral}
\biggl| \int_{0}^{\sigma} \chi_{U}(y_{1}-t^{a_{1}} x, y_{2}-t^{a_{2}} x^{d}-t^{a_{2}} )dx \biggl| \ge \sigma.
\end{equation}
Indeed, for each $x \in (0, \sigma)$ and $(y_{1},y_{2}) \in D_{t}$, it can be proved that $(y_{1} -t^{a_{1}}x, y_{2}-t^{a_{2}}x^{d} - t^{a_{2}}) \in T_{(r^{a_{1}}x,r^{a_{2}}x^{d}+r^{a_{2}})}$. Notice that $(y_{1} -t^{a_{1}} x +r^{a_{1}} x, y_{2}-t^{a_{2}}x^{d} - t^{a_{2}}  + r^{a_{2}} x^{d} +r^{a_{2}} ) = (z_{1} + (r^{a_{1}}-t^{a_{1}} )x, z_{2} + (r^{a_{2}}-t^{a_{2}} )x^{d})$, and
\[ |(z_{1} + (r^{a_{1}} -t^{a_{1}} )x, z_{2} + (r^{a_{2}} -t^{a_{2}} )x^{d}) \cdot e_{1}| \lesssim  \delta,\]
\[ |(z_{1} + (r^{a_{1}} -t^{a_{1}} )x, z_{2} + (r^{a_{2}} -t^{a_{2}} )x^{d}) \cdot e_{2}| \lesssim \delta \sigma^{-(d-1)}. \]
Therefore, $\chi_{U} (y_{1}-t^{a_{1}} x, y_{2}-t^{a_{2}} x^{d}-t^{a_{2}} ) = 1$ for each $x \in (0, \sigma)$ and inequality (\ref{lowerboundintegral}) follows.

By inequality (\ref{lowerboundintegral}), we have the following lower bound for $\sup_{t \in E \cap I}|A_{t}f|$,
\begin{equation}
\biggl\|\sup_{t \in E \cap I}|A_{t}f| \biggl\|_{L^{q}(\mathbb{R}^{2})} \ge \biggl\|\sup_{t \in E \cap I}|A_{t}f| \biggl\|_{L^{q}(\cup_{t \in E \cap I}D_{t})} \ge \sigma |\cup_{t \in E \cap I}D_{t}|^{\frac{1}{q}}.
\end{equation}
By the definition of the Assouad spectrum, $N(E \cap I, \delta) \ge \delta^{-\beta + \epsilon}$ for any $\epsilon > 0$. Combining  the fact that  $|t^{a_{2}}-r^{a_{2}}| \sim |t-r|$ for each $t \in [1,2]$ and $r \in [1,2]$, we get
\[|\cup_{t \in E \cap I}D_{t}| \ge \delta^{-\beta + \epsilon} \delta^{2} \sigma^{-(d-1)},\]
and
\begin{equation}\label{lowerboundlq}
\biggl\|\sup_{t \in E \cap I}|A_{t}f| \biggl\|_{L^{q}(\mathbb{R}^{2})}  \ge \sigma (\delta^{2-\beta + \epsilon} \sigma^{-(d-1)})^{\frac{1}{q}}.
\end{equation}
Inequalities (\ref{lowerboundlq}) and (\ref{upperboundlp}) imply that  if the $L^{p}(\mathbb{R}^{2}) \rightarrow L^{q}(\mathbb{R}^{2})$  estimate of $\sup_{t\in E} |A_{t}f|$ holds true, then
\[\sigma (\delta^{2-\beta + \epsilon} \sigma^{-(d-1)})^{\frac{1}{q}} \lesssim (\delta \sigma)^{\frac{1}{p}},\]
which yields
\begin{equation}
\frac{1}{d} \frac{\beta}{\tilde{\gamma}_{\theta}} + \biggl( 2 - \beta- \frac{d-1}{d} \frac{\beta}{\tilde{\gamma}_{\theta}} \biggl) \frac{1}{q} \ge \biggl( 1 + \frac{1}{d} \frac{\beta}{\tilde{\gamma}_{\theta}} \biggl) \frac{1}{p}
\end{equation}
because $\epsilon$ and $\delta$ can be sufficiently small, then we have finished the proof of inequality (\ref{N3}).

When $c =0$, we consider the local maximal operator
\[\sup_{t \in E}\biggl|\int_{0}^{1} f(y_1-t^{a_{1}}x,y_2-t^{a_{2}}x^{d})dx \biggl|.\]
By changing variables $x \rightarrow t^{\frac{a_{1}-a_{2}}{d-1}}x$, it is equivalent to consider the maximal operator
\[\sup_{t \in E}|A_{t}f(y)|:=\sup_{t \in E}\biggl|\int_{0}^{1} f(y_1-t^{\frac{da_{1}-a_{2}}{d-1}}x,y_2-t^{\frac{da_{1}-a_{2}}{d-1}}x^{d})dx \biggl|.\]
We need to prove that $\sup_{t \in E} |A_{t}|$ can not be $L^{p}(\mathbb{R}^{2}) \rightarrow L^{q}(\mathbb{R}^{2})$ bounded if $(\frac{1}{p},\frac{1}{q}) \notin \overline{\Delta}$. It has been proved in  {\cite[Theorem 2.7]{LWZ}} that, for fixed $t \in [1,2]$, the $L^{p}(\mathbb{R}^{2}) \rightarrow L^{q}(\mathbb{R}^{2})$ boundedness of $A_{t}$ implies  $\frac{1}{q} \le \frac{1}{p} \le \frac{2}{q} $ and $\frac{d+1}{p}-\frac{d+1}{q}-1 \le 0$.

%We first prove that the condition $\frac{d+1}{p}-\frac{d+1}{q}-1 \le 0$ is necessary. Let $f$ be the characteristic function on the set $D:=(-10 \delta, 10 \delta) \times (-10 \delta^{d}, 10 \delta^{d})$. For each $t \in [1,2]$, $y \in R:= (-\delta, \delta) \times (-\delta^{d},\delta^{d})$ and $x \in I:=(-\delta,\delta)$, $y-t(x,x^{d}) \in D$. Hence, the $L^{p}(\mathbb{R}^{2}) \rightarrow L^{q}(\mathbb{R}^{2})$ boundedness of $\sup_{t \in E} |A_{t}|$  implies $|R|^{\frac{1}{q}}|I| \lesssim |D|^{\frac{1}{p}}$, which requires $\frac{d+1}{p}-\frac{d+1}{q}-1\le 0$ since $\delta$ can be sufficiently small.

Below, we will prove that $\sup_{t \in E} |A_{t}|$ can not be $L^{p}(\mathbb{R}^{2}) \rightarrow L^{q}(\mathbb{R}^{2})$ bounded either $1+\frac{1-\beta}{q} < \frac{2}{p}$ or
 \begin{equation}\label{N4}
 \frac{1}{q} < \frac{2 \gamma + \beta}{2(2 \gamma + \beta)-\beta(3+2 \gamma) }\biggl( \frac{1}{p} -\frac{\beta}{2\gamma + \beta} \biggl).
 \end{equation}
Applying some transformations that do not change the $L^{p}(\mathbb{R}^{2}) \rightarrow L^{q}(\mathbb{R}^{2})$ norm of the above local maximal operator, we convert our problem to considering the following local maximal operator (which is also denoted by $\sup_{t \in E}|A_{t}|$ by abuse of notation),
\[\sup_{t \in E}|A_{t}f(y)|:=\sup_{t \in E}\biggl|\int_{-\frac{1}{2}}^{\frac{1}{2}} f\bigl(y_1-t^{\frac{da_{1}-a_{2}}{d-1}}(x+\frac{1}{2}),y_2-t^{\frac{da_{1}-a_{2}}{d-1}}\bigl(g(x)+ \frac{1-d}{2^d}\bigl) \bigl)dx \biggl|,\]
where $g(x)=x^{2}\biggl(\frac{d(d-1)}{2}(\frac{1}{2})^{d-2}+ \cdot \cdot \cdot + x^{d-2} \biggl)$. When $x$ is restricted to a small neighborhood of the origin, we have $|g(x)| \sim |x|^{2}$ and $|g^{\prime}(x)| \sim |x|$. Hence, the method we applied to obtain inequality (\ref{N2}) for the local maximal operator along the curve $\{(x,x^{2}+1): x \in (0,1)\}$ remains valid for the maximal operator considered here.
 Substituting $d =2$ in inequality (\ref{N2}) yields (\ref{N4}). For the same reason, the method in {\cite[Section 4.2]{AHRS}} can also be applied here to obtain the necessity of the condition $1 + \frac{1-\beta}{q} \ge \frac{2}{p}$.   $\hfill\square$

\section{The case when $da_{1} = a_{2}$: Proof of Theorem \ref{hhisomaintheorem}}
Theorem \ref{hhisomaintheorem} can be obtained from the following three lemmas. We notice that the positive constant $\epsilon^{\prime}$ can be found in Remark \ref{localyconstant+}.
\begin{lemma}\label{mainlemma1hom}
For each $k \ge \log(1/\epsilon)$ and $0 \le j \le mk+\epsilon^{\prime} j$, there exists $\epsilon(p,q)>0$ such that
\begin{equation}\label{homE1}
\biggl\|\sup_{t \in [1,2]} |\widetilde{A_{t,j}^{k}}f| \biggl\|_{L^{q}(\mathbb{R}^{2})} \lesssim 2^{-\epsilon(p,q)j}  \|f\|_{L^{p}(\mathbb{R}^{2})}
\end{equation}
whenever $(1/p,1/q)$ satisfies $1/2p < 1/q \le 1/p$ and $1/q > 2/p-1$.
\end{lemma}
\begin{lemma}\label{mainlemma2hom}
For $1 \le i \le 3$, $k \ge \log(1/\epsilon)$ and $j > mk + \epsilon^{\prime} j$, there holds
\begin{equation}\label{homE2}
\biggl\|\sup_{t \in E} |\widetilde{A_{t,j}^{k}}f| \biggl\|_{L^{q_{i}}(\mathbb{R}^{2})} \lesssim 2^{-\frac{mk\beta}{q_{i}}}  \|f\|_{L^{p_{i}}(\mathbb{R}^{2})}.
\end{equation}
What's more, there exists $\epsilon_{1}>0$ which depends on $\beta$ such that
\begin{equation}\label{homE3}
\biggl\|\sup_{t \in E} |\widetilde{A_{t,j}^{k}}f| \biggl\|_{L^{2}(\mathbb{R}^{2})} \lesssim 2^{-\frac{mk\beta}{2}-\epsilon_{1}j}  \|f\|_{L^{2}(\mathbb{R}^{2})}.
\end{equation}
\end{lemma}
\begin{lemma}\label{mainlemma3hom}
When $\gamma \in (0,1)$,  for each $k \ge \log(1/\epsilon)$, $j > mk+\epsilon^{\prime} j$, there holds
\begin{equation}\label{homE4}
\biggl\|\sup_{t\in E}|\widetilde{A_{t,j}^{k}}f |\biggl\|_{L^{q_{4}}({\mathbb{R}}^2)} \lesssim_{\gamma} j \|f\|_{L^{p_{4}}({\mathbb{R}}^2)},
\end{equation}
where the implied constant depends on $\gamma$, but does not depend on $j$ and $k$.
\end{lemma}

%We decompose $A_{t}$ (which is given by (\ref{averagefinite3})) as at the beginning of Section \ref{istropic} and define $\widetilde{A_{t,j}^{k}}$, $\widetilde{\mathcal{M}^{0,k}}$ and $\widetilde{\mathcal{M}_{j}^{k}}$.
%According to Remark \ref{localyconstant+}, we consider $j > mk + \epsilon j$ and $j \le mk + j\epsilon$, respectively. Similar with the proof of estimates (E1-E2) in {\cite[Theorem 1.4]{LWZ}}, when $j \le mk + j\epsilon$, we can get
%whenever $(1/p,1/q)$ satisfies $1/2p < 1/q \le 1/p$ and $1/q > 2/p-1$.

We briefly explain how to prove Theorem \ref{hhisomaintheorem} by Lemmas \ref{mainlemma1hom}-\ref{mainlemma3hom}. By interpolation (\ref{homE3}) with (\ref{homE2})  for $1 \le i \le 3$, respectively, and combining the inequality (\ref{homE4}),  there exists  a constant $\epsilon(p,q)>0$ such that
\[\biggl\|\sup_{t\in E} |\widetilde {A_{t,j}^{k}}f|\biggl\|_{L^{q}(\mathbb{R}^{2})} \lesssim 2^{-\epsilon(p,q)j}\|f\|_{L^{p}(\mathbb{R}^{2})} \]
holds whenever $j>mk+\epsilon^{\prime} j$ and $(\frac{1}{p},\frac{1}{q}) \in \mathcal{R}^{2}(\beta, \gamma)\setminus \{(0,0)\}$.
Then, combining the results from Lemma \ref{mainlemma1hom} and isometric transformation, we have
\begin{align*}
\biggl\|\sup_{t\in E}|A_{t,j}^{k}f|\biggl\|_{L^{q}(\mathbb{R}^{2})}
&\le \sum_{k \ge \log(1/\epsilon)} \sum_{j \ge 0} 2^{(d+1)(\frac{1}{p}-\frac{1}{q})k-k} \biggl\|\sup_{t \in [1,2]} |\widetilde{A_{t,j}^{k}}f| \biggl\|_{L^{q}(\mathbb{R}^{2})}\nonumber\\
&\le \sum_{k \ge \log(1/\epsilon)} \sum_{j \ge 0} 2^{(d+1)(\frac{1}{p}-\frac{1}{q})k-k} 2^{-\epsilon(p,q)j} \|f\|_{L^{p}(\mathbb{R}^{2})} \lesssim \|f\|_{L^{p}(\mathbb{R}^{2})}
\end{align*}
provided that $\frac{d+1}{p}-\frac{d+1}{q} - 1<0$, which implies Theorem \ref{hhisomaintheorem}. Next, we will prove Lemma \ref{mainlemma1hom} and Lemma \ref{mainlemma2hom} in Subsection \ref{proofofhom12}, and Lemma \ref{mainlemma3hom} in Subsection \ref{proofofhom3}.

%Next we concentrate on the case when $j > mk + j\epsilon$. Combining with the proof of {\cite[Theorem 1.4]{LWZ}} and the argument in Subsection \ref{proofofmainlemma1}, it is easy to prove that $\widetilde{\mathcal{M}^{0,k}}$ and $\widetilde{\mathcal{M}_{j}^{k}}$ satisfy Lemma \ref{basicestimate} (with $c=0$). Comparing with the proof of Theorem \ref{maintheorem1}, the main difference lies at the point $(\frac{1}{p_4},\frac{1}{q_4})=(\frac{2}{3+2\gamma}, \frac{1}{3+2\gamma})$. We show some details for the proof of the model case when $d=2$, $a_{1} = 1$ and $a_{2} = 2$, since the main idea can be applied to more general case.
\subsection{Proof of Lemma \ref{mainlemma1hom} and Lemma \ref{mainlemma2hom}}\label{proofofhom12}
We first explain how to prove Lemma \ref{mainlemma1hom}. When $j =0$, Lemma \ref{mainlemma1hom} follows directly from
Soboelev's embedding and inequality (\ref{A1}). Since the operator $R_{t,j}^{k}$ satisfies (\ref{remainderterm}), we only need to prove Lemma \ref{mainlemma1hom} for the operator $A_{t,j}^{k}$ when $j \ge 1$. Noting that $j \le mk + \epsilon^{\prime} j$ in Lemma  \ref{mainlemma1hom}, Sobolev's embedding  and inequality (\ref{L2}) imply the validity of
\[\biggl\|\sup_{t \in [1,2]}|A_{t,j}^{k}f| \biggl\|_{L^{2}(\mathbb{R}^{2})} \lesssim 2^{-\frac{j}{2} + \epsilon^{\prime} j} \|f\|_{L^{2}(\mathbb{R}^{2})}.\]
Using the same method, we can derive
 \[\biggl\|\sup_{t \in [1,2]}|A_{t,j}^{k}f| \biggl\|_{L^{1}(\mathbb{R}^{2})} \lesssim 2^{ \epsilon^{\prime} j} \|f\|_{L^{1}(\mathbb{R}^{2})}\]
 from inequality (\ref{L1}). Based on inequalities (\ref{Linfty}) and (\ref{Linfty1}), we can establish the following estimates
  \[\biggl\|\sup_{t \in [1,2]}|A_{t,j}^{k}f| \biggl\|_{L^{\infty}(\mathbb{R}^{2})} \lesssim  \|f\|_{L^{\infty}(\mathbb{R}^{2})}\]
 and .
  \[\biggl\|\sup_{t \in [1,2]}|A_{t,j}^{k}f| \biggl\|_{L^{\infty}(\mathbb{R}^{2})} \lesssim  2^{j} \|f\|_{L^{1}(\mathbb{R}^{2})},\]
 respectively. Thus, Lemma \ref{mainlemma1hom} can be obtained by interpolation.

Next, we will provide a brief explanation of the proof for Lemma \ref{mainlemma2hom}. By selecting any interval $I$ of length $2^{-j+mk}$ that is contained in $[1,2]$, according to inequality (\ref{L2}), inequality (\ref{remainderterm}), and Remark \ref{localyconstant+}, we have
\begin{align}
\biggl\| \sup_{t \in I}|\widetilde{A_{t,j}^{k}}| \biggl\|_{L^{2}(\mathbb{R}^{2})} &\lesssim_{N}  2^{j-mk}  \int_{1/2}^{4} \frac{  \|\widetilde{A_{s,j}^{k}}f  \|_{L^{2}(\mathbb{R}^{2})} }{(1+ 2^{j-mk} dist(s,I))^{N}}ds
 +  2^{-Nj} \|f\|_{L^{2}(\mathbb{R}^{2})} \nonumber\\
&\lesssim_{N} 2^{-\frac{j}{2}}2^{j-mk} \int_{1/2}^{4}  (1+ 2^{j-mk} dist(s,I))^{-N}ds  \|f  \|_{L^{2}(\mathbb{R}^{2})} + 2^{-Nj} \|f\|_{L^{2}(\mathbb{R}^{2})} \nonumber\\
&\lesssim_{N} 2^{-\frac{j}{2}} \|f\|_{L^{2}(\mathbb{R}^{2})}. \nonumber
 \end{align}
When combined  with (\ref{defofbeta}), we obtain
\[\biggl\|\sup_{t\in E}|\widetilde{A_{t,j}^{k}}f| \biggl\|_{L^{2}({\mathbb{R}}^2)}\lesssim 2^{-\frac{mk \beta}{2}+\frac{\beta j}{2}} 2^{-\frac{j}{2}}\|f\|_{L^2({\mathbb{R}}^2)}.\]
Then (\ref{homE2}) follows since $\beta \in (0,1)$. Using a similar proof, we can derive  from inequality (\ref{L1}) that
\[\biggl\|\sup_{t\in E}|\widetilde{A_{t,j}^{k}}f| \biggl\|_{L^{1}({\mathbb{R}}^2)}\lesssim  2^{-mk\beta+\beta j}\|f\|_{L^1({\mathbb{R}}^2)}.\]

Furthermore, inequalities (\ref{Linfty}) and (\ref{Linfty1}) imply
\[
\biggl\|\sup_{t\in E}|\widetilde{A_{t,j}^{k}}f| \biggl\|_{L^{\infty}({\mathbb{R}}^2)}\lesssim \|f\|_{L^\infty({\mathbb{R}}^2)},
\]
and
\[
\biggl\|\sup_{t\in E}|\widetilde{A_{t,j}^{k}}f| \biggl\|_{L^{\infty}({\mathbb{R}}^2)}\lesssim 2^{j}\|f\|_{L^1({\mathbb{R}}^2)},
\]
respectively.
Finally, we can obtain inequality (\ref{homE2}) through interpolation among the above results.

\subsection{Proof of Lemma \ref{mainlemma3hom}}\label{proofofhom3}
%The dominating contribution term is given by
%\begin{equation}\label{mainterm3}
%A_{t,j}^kf(y):=\frac{1}{(2\pi)^2}\int_{{\mathbb{R}}^2}e^{i(\xi \cdot
%y-t^2\xi_2\tilde{\Phi}(s,\delta))}\chi_{k,m}(\frac{\xi_1}{t\xi_2})
%\frac{A_{k,m}(\delta_{t} \xi)}{(1+|\delta_t\xi|)^{1/2}}\beta(2^{-j}|\delta_t \xi|)\hat{f}(\xi)d\xi,
%\end{equation}
%where  $\chi_{k,m}$ is a smooth function supported in $[c_{k,m},\widetilde{c_{k,m}}]$, for certain non-zero constants $c_{k,m}$ and $\widetilde{c_{k,m}}$ dependent only on $k$ and $m$. $A_{k,m}$ is a symbol of order zero in $\xi$ and $\{A_{k,m}(\delta_t\xi)\}_{k,m}$ is contained in a bounded subset of symbol of order zero.
%In the phase function
%\begin{equation}\label{Phase}
%-t^2\xi_2\tilde{\Phi}(s,\delta)=\frac{\xi_1^2}{2\xi_2}+(-1)^{m+1}\frac{{\phi}^{(m)}(0)}{m!}\delta^{m}\frac{\xi_1^{m+2}}{t^m\xi_2^{m+1}} + R(t,\xi,\delta),
%\end{equation}
%where $\delta = 2^{-k}$, $R(t,\xi,\delta)$ is homogeneous of degree one in $\xi$.
%Note that $-t^2\xi_2\tilde{\Phi}(s,\delta)$ can be considered as a small perturbation of  \[\frac{\xi_1^2}{2\xi_2}+(-1)^{m+1}\frac{{\phi}^{(m)}(0)}{m!}\delta^{m}\frac{\xi_1^{m+2}}{t^m\xi_2^{m+1}}. \]

For each $j$, $k$ with $j > mk + \epsilon^{\prime} j$, let $\mathcal{I}_{j,k}(E)$ denote the collection of intervals of the form $[l2^{-j+mk}, (l+1)2^{-j+mk}]$, $l \in \mathbb{Z}$, which intersect $E$. Let $\widetilde{E}_{j,k} = \{t_{l}\}_{l}$ be the left endpoints of the intervals in $\mathcal{I}_{j,k}(E)$.
%\begin{lemma}
%When $\gamma \in (0,1)$,  for each $j \ge mk+\epsilon j$, there holds
%\begin{equation}\label{weakestimate}
%\|A_{t_{l},j}^{k}f \|_{L^{3+2\gamma}({\mathbb{R}}^2 \times E_{j,k})} \lesssim_{\gamma} j \|f\|_{L^{\frac{3+2\gamma}{2}}({\mathbb{R}}^2)},
%\end{equation}
%where the implied constant depends on $\gamma$, but never depends on $j$ and $k$.
%\end{lemma}
As in the proof of Lemma \ref{smallgamma}, the case $\gamma \in (0,1/2)$ can be proved by interpolation between
\begin{equation}\label{smallgamma2++}
\|A_{t_{l},j}^{k}f \|_{L^{\infty}({\mathbb{R}}^2 \times \widetilde{E}_{j,k})} \lesssim    2^{j} \|f\|_{L^1({\mathbb{R}}^2)},
\end{equation}
and
\begin{equation}\label{smallgamma1++}
\|A_{t_{l},j}^{k}f \|_{L^{2(1+2\gamma), \infty}({\mathbb{R}}^2 \times \widetilde{E}_{j,k})} \lesssim    2^{-\frac{1-2\gamma}{2(1+2\gamma)}j} \|f\|_{L^2({\mathbb{R}}^2)}.
\end{equation}
 Inequality (\ref{smallgamma2++}) holds because inequality (\ref{Linfty1}) remains valid for each $t \in [1,2]$, and inequality (\ref{smallgamma1++}) follows from Theorem  \ref{smallgamma+}, which can be found in the appendix of this article. Next we  state the proof for $\gamma \in [1/2,1)$.

When $a_{1}=1$, $a_{2}=d$, the phase function in the operator $A_{t,j}^{k}$ becomes
 \[-\tilde{\Phi}(t,\xi,\delta)=(d-1) \xi_{2} \biggl(-\frac{\xi_{1}}{d\xi_{2}}\biggl)^{\frac{d}{d-1}}-t^{-m} \xi_{2}\frac{\delta^m\phi^{(m)}(0)}{ m!}
\biggl(-\frac{\xi_1}{d\xi_2}\biggl)^{\frac{d+m}{d-1}}+ R(t,\xi,\delta,d),\]
 with the leading term involving $t$ being $t^{-m} \xi_{2}\frac{\delta^m\phi^{(m)}(0)}{ m!}
\biggl(-\frac{\xi_1}{d\xi_2}\biggl)^{\frac{d+m}{d-1}}$. The following analysis related to the phase function
%and the factor $\delta^{m}$  will affect the arguments on the orthogonality of the Fourier integral operator as concluded in Proposition \ref{keylemma1}. In %addition, the calculations in Proposition \ref{keylemma1} and Proposition \ref{keylemma2}
depends on the structure of each term in the phase function, which is represented as $c_{\alpha}\xi_{2}(-\frac{\xi_{1}}{\xi_{2}})^{\alpha}$, regardless of the specific value of $\alpha$.
%regardless of the specific value of the index $a$.
Therefore, in order to clarify what occurs when $da_{1}=a_{2}$, we will take $d=2$ in the following to illustrate the proof.

Denote
\begin{align}
F_{j}^{k}(f\otimes f)(y,t_{l}) &=  \int_{\mathbb{R}^{2}} \int_{\mathbb{R}^{2}}e^{i[y \cdot (\xi + \zeta) + \frac{\xi_1^2}{2\xi_2} + \frac{\zeta_1^2}{2\zeta_2}+(-1)^{m+1}\frac{{\phi}^{(m)}(0)}{m!}\delta^{m}( \frac{\xi_1^{m+2}}{t_{l}^m\xi_2^{m+1}} +\frac{\zeta_1^{m+2}}{t_{l}^m\zeta_2^{m+1} }) + R(t_{l},\xi,\zeta, \delta)]} \nonumber\\
&\quad \quad \times a_{j}(t_{l},\xi, \zeta) \hat{f}(\xi)\hat{f}(\zeta)d\xi d\zeta, \nonumber
\end{align}
where $a_{j}(t_{l}, \xi, \zeta) =2^{j}\chi_{k,m}(\frac{\xi_1}{t_{l}\xi_2})
\frac{A_{k,m}(\delta_{t_{l}}\xi)}{(1+|\delta_{t_{l}}\xi|)^{1/2}}\psi(2^{-j}|\delta_{t_{l}} \xi|) \chi_{k,m}(\frac{\zeta_1}{t_{l}\zeta_2})
\frac{A_{k,m}(\delta_{t_{l}}\zeta)}{(1+|\delta_{t_{l}}\zeta|)^{1/2}}\psi(2^{-j}|\delta_{t_{l}} \zeta|)$. By the same argument as in the proof of Lemma \ref{smallgamma}, we only need to prove the following estimate
\begin{equation}\label{blinear3+}
\|F_{j}^{k}(f\otimes f)(y,t_{l})\|_{L^{ \frac{3+2\gamma}{2}}({\mathbb{R}}^2 \times \widetilde{E}_{j,k})} \lesssim_{\gamma} j 2^{j} \|f\|^2_{L^{\frac{3+2\gamma}{2}}({\mathbb{R}}^2)}.
\end{equation}
By similar arguments as in the proof of inequality (\ref{blinear3}) based on Lemma \ref{bilinear3lemma1} and Lemma \ref{bilinear3lemma2}, (\ref{blinear3+}) can be obtained from Lemma \ref{bilinear3lemma1+} and Lemma \ref{bilinear3lemma2+} below.

\begin{lemma}\label{bilinear3lemma1+}
For each $j > mk + \epsilon^{\prime} j$, we have
\begin{equation}
\|F_{j}^{k}(f\otimes f)(y,t_{l})\|_{L^{ \infty}({\mathbb{R}}^2 \times \widetilde{E}_{j,k})} \lesssim 2^{j} \int_{\mathbb{R}^{2}}\sup_{(w_{2},z_{2})\in \mathbb{R}^{2}}|f(w_{1},w_{2})||f(z_{1},z_{2})|d w_{1}dz_{1}.
\end{equation}
\end{lemma}

\begin{lemma}\label{bilinear3lemma2+}
For each $j > mk + \epsilon^{\prime} j$, there holds
\begin{equation}\label{blinear3L22+}
\|F_{j}^{k}(f\otimes f)(y,t_{l})\|_{L^{2}({\mathbb{R}}^2 \times \widetilde{E}_{j,k})} \lesssim_{\gamma} j 2^{j} \biggl( \int_{\mathbb{R}^{2}} \int_{\mathbb{R}^{2}}|f(w)|^{2}|f(z)|^{2}|w-z|^{-(2\gamma-1)}dwdz \biggl)^{1/2}.
\end{equation}
\end{lemma}
Employing the same argument as the proof of Lemma \ref{bilinear3lemma1}, we can obtain Lemma \ref{bilinear3lemma1+}. We state the proof of Lemma \ref{bilinear3lemma2+} in the rest of this section.

\textbf{Proof of Lemma \ref{bilinear3lemma2+}.} By a change of variables and Plancherel's theorem, it suffices to show
\begin{equation}\label{blinear3L22+}
\|T_{j}^{k}(\mathcal{F}, b)(\xi,t_{l})\|_{L^{2}({\mathbb{R}}^2 \times \widetilde{E}_{j,k})} \lesssim_{\gamma} j 2^{j} \biggl( \int_{\mathbb{R}^{2}} \int_{\mathbb{R}^{2}}|f(w)|^{2}|f(z)|^{2}|w-z|^{-(2\gamma-1)}dwdz \biggl)^{1/2},
\end{equation}
where $b$ is a non-negative bump function such that supp $b \subset B(0,C_{k,m})$, $\hat{\mathcal{F}}(\xi-\zeta, \zeta) =\hat{f}(\xi - \zeta)\hat{f}(\zeta)$ and
\begin{align}
T_{j}^{k}(\mathcal{F}, b)(\xi,t_{l})&=  \int_{\mathbb{R}^{2}}e^{ i[\frac{(\xi_{1}-\zeta_{1})^{2}}{2(\xi_{2}-\zeta_{2})} + \frac{\zeta_1^2}{2\zeta_2}+(-1)^{m+1}\frac{{\phi}^{(m)}(0)}{m!}\delta^{m}( \frac{(\xi_{1}-\zeta_{1})^{m+2}}{t_{l}^m(\xi_{2}-\zeta_{2})^{m+1}} +\frac{\zeta_1^{m+2}}{t_{l}^m\zeta_2^{m+1} }) + R(t_{l},\xi-\zeta,\zeta, \delta)]} \nonumber\\
&\quad \quad \times b \biggl(\frac{\xi_{1}}{\xi_{2}}-\frac{\zeta_{1}}{\zeta_{2}} \biggl)a_{j}(t_{l},\xi-\zeta, \zeta) \hat{\mathcal{F}}(\xi-\zeta,\zeta) d\zeta. \nonumber
\end{align}
Then we decompose $T_{j}^k(\mathcal{F},b)(\xi,t_{l}) = \sum_{h=1}^{(j-mk)/2}T_{j}^k(\mathcal{F},b_{h})(\xi,t_{l})$, where $b_{h}$'s are even and supp $b_{h} \subset \{ \tau \in \mathbb{R}: |\tau| \sim 2^{-h} \}$ for $1 \le h < \frac{(j-mk)}{2}$ and supp $b_{\frac{j-mk}{2}} \subset \{ \tau \in \mathbb{R}: |\tau| \lesssim 2^{-(j-mk)/2} \}$. For fixed $h$, we further decompose $\mathcal{F} = \sum_{n=0}^{\infty} \mathcal{F}_{-h+n}$ with supp $\mathcal{F}_{-h} \subset \{(w,z) \in \mathbb{R}^{2} \times \mathbb{R}^{2}: |z-w| \lesssim 2^{-h}\}$ and supp $\mathcal{F}_{-h+n} \subset \{(w,z) \in \mathbb{R}^{2} \times \mathbb{R}^{2}: |z-w| \sim 2^{-h+n}\}$, $n \ge 1$. We have the following two estimates.

\begin{proposition}\label{keylemma1+}
For each $j > mk + \epsilon^{\prime} j$  and $1 \le h \le \frac{j-mk}{2}$, we have
\begin{equation}
\|T_{j}^{k}(\mathcal{F}_{-h}, b_{h})\|_{L^{2}(\mathbb{R}^{2} \times \widetilde{E}_{j,k})} \lesssim 2^{j-\frac{h}{2}} 2^{\gamma h} \|\mathcal{F}_{-h}\|_{L^{2}(\mathbb{R}^{2} \times \mathbb{R}^{2})}.
\end{equation}
\end{proposition}

\begin{proposition}\label{keylemma2+}
For each $j > mk + \epsilon^{\prime} j$, $n \ge 1$ and $1 \le h \le \frac{j-mk}{2}$, there holds
\begin{equation}
\|T_{j}^{k}(\mathcal{F}_{-h+n}, b_{h})\|_{L^{2}(\mathbb{R}^{2} \times \widetilde{E}_{j,k})} \lesssim 2^{\frac{3}{2}h-n} 2^{\frac{\gamma j}{2}}2^{-\frac{\gamma mk}{2}} \|\mathcal{F}_{-h+n}\|_{L^{2}(\mathbb{R}^{2} \times \mathbb{R}^{2})}.
\end{equation}
\end{proposition}

If the above two propositions hold true, then
\begin{align}
&\|T_{j}^{k}(\mathcal{F}, b)(y,t_{l})\|_{L^{2}({\mathbb{R}}^2 \times \widetilde{E}_{j,k})} \nonumber\\
&\le  \sum_{h=1}^{\frac{j-mk}{2}}
\|T_{j}^{k}(\mathcal{F}_{-h}, b_{h})(y,t_{l})\|_{L^{2}({\mathbb{R}}^2 \times \widetilde{E}_{j,k})} +  \sum_{h=1}^{\frac{j-mk}{2}} \sum_{n=1}^{\infty} \|T_{j}^{k}(\mathcal{F}_{-h+n}, b_{h})(y,t_{l})\|_{L^{2}({\mathbb{R}}^2 \times \widetilde{E}_{j,k})} \nonumber\\
&\le  \sum_{h=1}^{\frac{j-mk}{2}} 2^{j}\biggl( \int_{\mathbb{R}^{2}} \int_{\mathbb{R}^{2}}|f(w)|^{2}|f(z)|^{2}|w-z|^{-(2\gamma-1)}dwdz \biggl)^{1/2} \nonumber\\
&\quad \quad + \sum_{h=1}^{\frac{j-mk}{2}} \sum_{n=1}^{\infty} 2^{\frac{3}{2}h-n} 2^{\frac{\gamma j}{2}}2^{-\frac{\gamma mk}{2}} 2^{(-h+n)(\gamma - \frac{1}{2})}  \biggl( \int_{\mathbb{R}^{2}} \int_{\mathbb{R}^{2}}|f(w)|^{2}|f(z)|^{2}|w-z|^{-(2\gamma-1)}dwdz \biggl)^{1/2} \nonumber\\
 % &\le  j 2^{j}\biggl( \int_{\mathbb{R}^{2}} \int_{\mathbb{R}^{2}}|f(w)|^{2}|f(z)|^{2}|w-z|^{-(2\gamma-1)}dwdz \biggl)^{1/2} \nonumber\\
 % &\quad \quad + \sum_{h=1}^{\frac{j-mk}{2}}  2^{(2-\gamma)h} 2^{\frac{\gamma j}{2}}2^{-\frac{\gamma mk}{2}} \biggl( \int_{\mathbb{R}^{2}} \int_{\mathbb{R}^{2}}|f(w)|^{2}|f(z)|^{2}|w-z|^{-(2\gamma-1)}dwdz \biggl)^{1/2} \nonumber\\
&\lesssim  j 2^{j}\biggl( \int_{\mathbb{R}^{2}} \int_{\mathbb{R}^{2}}|f(w)|^{2}|f(z)|^{2}|w-z|^{-(2\gamma-1)}dwdz \biggl)^{1/2}.
\end{align}
$\hfill\square$

Now we are reduced to proving Proposition \ref{keylemma1+} and \ref{keylemma2+}.

\textbf{Proof of Proposition \ref{keylemma1+}.} Noticing that  $t_l-t_{l'}\geq 2^{-j+mk}$ if $l\not=l'$, then the case when $h= \frac{j-mk}{2}$ follows from the Cauchy-Schwarz inequality and Plancherel's theorem. To prove the case when $1\le h<\frac{j-mk}{2}$, we define
\[K(t_{l},t_{l^{\prime}}, \xi) =  \int_{\mathbb{R}^{2}}e^{i   \Phi_{t_{l},t_{l^{\prime}}, \xi}(\zeta)}b_{h} \biggl(\frac{\xi_{1}}{\xi_{2}}-\frac{\zeta_{1}}{\zeta_{2}} \biggl)a_{j,t_{l},t_{l^{\prime}}}(\xi, \zeta) d\zeta,\]
where  $a_{j,t_{l},t_{l^{\prime}}}(\xi, \zeta)= a_{j}(t_{l},\xi-\zeta, \zeta) a_{j}(t_{l^{\prime}},\xi-\zeta, \zeta) $, and
\begin{align}
\Phi_{t_{l},t_{l^{\prime}},\xi}(\zeta) &=(t_{l}^{-m} - t_{l^{\prime}}^{-m}) (-1)^{m+1}\frac{{\phi}^{(m)}(0)}{m!}\delta^{m} \biggl(\frac{(\xi_{1}-\zeta_{1})^{m+2}}{(\xi_{2}-\zeta_{2})^{m+1}} + \frac{\zeta_1^{m+2}}{\zeta_2^{m+1} }\biggl)+R(t_{l},t_{l^{\prime}},\xi-\zeta,\zeta, \delta).\nonumber
\end{align}
By the similar argument as in the proof of Proposition \ref{keylemma1}, changing variables $\zeta= B\eta$ with $B=\left(
\begin{array} {lcr}
2^{-h}\xi_{2} & \xi_{1}  \\
0 & \xi_{2}
\end{array}
\right)$,  we get
\[K(t_{l},t_{l^{\prime}}, \xi) = \xi_{2}^{2} 2^{-h}  \int_{\mathbb{R}^{2}}e^{i \Phi_{h,t_{l},t_{l^{\prime}}, \xi}(\eta)}\tilde{b} \biggl(\frac{\eta_{1}}{\eta_{2}} \biggl)a_{j,h,t_{l},t_{l^{\prime}},\xi}(\eta) d\eta,\]
where $\Phi_{h,t_{l},t_{l^{\prime}},\xi}(\eta) = \Phi_{t_{l},t_{l^{\prime}},\xi}(B\eta)$, $a_{j,h,t_{l},t_{l^{\prime}},\xi}(\eta) = a_{j, t_{l},t_{l^{\prime}}}(\xi, B\eta)  $, $\tilde{b}$ is a bump function supported in $\{x \in \mathbb{R}: |x| \sim 1\}$.
By Taylor's expansion, $\Phi_{h,t_{l},t_{l^{\prime}}, \xi}(\eta)$ satisfies
\[|\nabla_{\eta} \Phi_{h, t_{l},t_{l^{\prime}},\xi}(\eta)| \sim 2^{j-mk-2h}|t_{l}^{-m}-t_{l^{\prime}}^{-m}|,\]
and
\[|D_{\eta}^{\alpha} \Phi_{h, t_{l},t_{l^{\prime}},\xi}(\eta)| \lesssim 2^{j-mk-2h} |t_{l}^{-m}-t_{l^{\prime}}^{-m}|, \quad \quad |\alpha| \ge 2.\]
Since $a_{j,h,t_{l},t_{l^{\prime}},\xi}$ is a symbol of order zero and $\widetilde{E}_{j,k} \subset [1,2]$, integration by parts implies  that
\begin{equation}\label{orthognality1+}
|K(t_{l},t_{l^{\prime}}, \xi)| \le 2^{2j-h} \frac{C_{N}}{(1+2^{j-mk-2h}|t_{l}-t_{l^{\prime}}|)^{N}}, \quad \quad N \gg 1
\end{equation}
for each $\xi$ and  $t_{l},t_{l^{\prime}} \in \widetilde{E}_{j,k}$. Notice that for each interval with length $\lesssim 2^{-j +mk +2h}$,
\[Card(\widetilde{E}_{j,k}\cap I) \lesssim 2^{2\gamma h}. \]
Then we finish the proof of Proposition \ref{keylemma1+} by the similar argument as in {\cite[Proposition 4.2]{Roos}}, see also Proposition \ref{keylemma1} in this paper.
$\hfill\square$

\textbf{Proof of Proposition \ref{keylemma2+}.}  Denote $\tilde{a}_{t_{l},j,h,\xi}( \zeta)= b_{h}(\frac{\xi_{1}}{\xi_{2}}-\frac{\zeta_{1}}{\zeta_{2}})a_{j}(t_{l}, \zeta - \xi, \zeta)$,
\begin{align}
\Phi_{\xi,w}(\zeta)&= w \cdot \zeta + \frac{(\xi_{1}-\zeta_{1})^{2}}{2(\xi_{2}-\zeta_{2})} + \frac{\zeta_{1}^{2}}{2\zeta_{2}}\nonumber\\
&\quad  + (-1)^{m+1}\frac{{\phi}^{(m)}(0)}{m!}\delta^{m} \biggl(\frac{(\xi_{1}-\zeta_{1})^{m+2}}{t_{l}^{m}(\xi_{2}-\zeta_{2})^{m+1}} + \frac{\zeta_1^{m+2}}{t_{l}^{m}\zeta_2^{m+1} }\biggl) + R(t_{l},\xi-\zeta,\zeta, \delta),\nonumber
\end{align}
which can be viewed as a small perturbation of $ w \cdot \zeta + \frac{(\xi_{1}-\zeta_{1})^{2}}{2(\xi_{2}-\zeta_{2})} + \frac{\zeta_{1}^{2}}{2\zeta_{2}}$.
Then as in the proof of Proposition \ref{keylemma2}, there holds
\[|\nabla_{\zeta}\Phi_{\xi,w}(\zeta)| \sim 2^{-h+n}\]
and $|D^{\alpha}_{\zeta}\Phi_{\xi,w}(\zeta)| \lesssim 2^{-j}$ for each $\alpha$ such that $|\alpha| = 2$. Moreover,
$|\nabla_{\zeta} \tilde{a}_{t_{l},j,h,\xi}(\zeta)| \lesssim 2^{h-j}$. Let
\[a_{t_{l},j,h, \xi}(\zeta,w) = i \frac{\langle \nabla_{\zeta} \tilde{a}_{t_{l},j,h, \xi}, \nabla_{\zeta} \Phi_{\xi,w} \rangle}{|\nabla_{\zeta} \Phi_{\xi,w}|^{2}} - 2i \tilde{a}_{t_{l},j,h, \xi} \frac{\langle H_{\zeta}(\Phi_{\xi,w}) \nabla_{\zeta} \Phi_{\xi,w}, \nabla_{\zeta} \Phi_{\xi,w} \rangle}{|\nabla_{\zeta} \Phi_{\xi,w}|^{4}}. \]
Then we have
\[| a_{t_{l},j,h, \xi}(\zeta,w) | \le 2^{2h-j-n},\]
and
\[|D_{\zeta}^{\beta}D_{w}^{\alpha} [a_{t_{l},j,h, \xi}(\zeta,w)] | \le 2^{2h-j-n} 2^{(h-n)|\alpha|} 2^{(h-j)|\beta|}   \]
for each multi-index $\alpha$ and $\beta$. Since $h \le \frac{j-mk}{2}$ and $|\zeta| \sim 2^{j}$, we have
\begin{equation}
|D_{\zeta}^{\beta}D_{w}^{\alpha} [a_{t_{l},j,h, \xi}(\zeta,w)] | \le 2^{2h-j-n} (1+|\zeta|)^{-\frac{1}{2}(|\beta|-|\alpha|)}
\end{equation}
for each multi-index $\alpha$ and $\beta$. Then it follows from \cite{CV} that
\begin{equation}\label{offdianol+}
\biggl\| \int_{\mathbb{R}^{2}}e^{i w \cdot \zeta} a_{t_{l}, j,h,\xi}(\zeta,w) \psi(2^{h-n}w) \hat{\mathcal{F}}^{y}_{-h+n}(\xi, \xi +w)dw \biggl\|_{L_{\zeta}^{2}(\mathbb{R}^{2})} \lesssim 2^{2h-j-n} \|\hat{\mathcal{F}}_{-h+n}(\xi, \xi + \zeta)\|_{L_{\zeta}^{2}(\mathbb{R}^{2})},
\end{equation}
where $\psi$ is a bump function supported in $\{x \in \mathbb{R}: |x| \sim 1\}$.
Then Proposition \ref{keylemma2+} follows from the similar arguments as in the proof of Proposition \ref{keylemma2}, and the fact that $Card(\widetilde{E}_{j,k}) \lesssim 2^{\gamma j- \gamma mk}$. %More details can be found in  the proof of {\cite[Proposition 4.3]{Roos}}.
$\hfill\square$

\section{Appendix}
This appendix is dedicated to the proof of inequalities (\ref{smallgamma1+}) and (\ref{smallgamma1++}), and Lemma \ref{kernel estimate3} below will be used in the proof, which is derived from Lemma 2.6 in \cite{LWZ}.
\begin{lemma}\label{kernel estimate3}
Let $\lambda \gg 1$, $\tilde{\chi}$ and $\psi$ be smooth cutoff functions supported in $[1,2]$. If $\Psi$ is a smooth function that is defined on $[1,2]$, and $|\Psi^{\prime \prime}(s)| \sim 1$ for each $s \in [1,2]$, then we have the uniform pointwise estimate
\[ \biggl|\int_{\mathbb{R}^{2}}e^{i\lambda[y\cdot\eta + \eta_{2}\Psi(\frac{\eta_{1}}{\eta_{2}})]} \tilde{\chi}(\frac{\eta_{1}}{\eta_{2}})\psi(|\eta|)d\eta \biggl| \lesssim \lambda^{-\frac{1}{2}}, \quad y \in \mathbb{R}^{2},\]
where the implied constant is independent of $\lambda$ and $y$.
\end{lemma}

Inequalities (\ref{smallgamma1+}) and (\ref{smallgamma1++}) can be obtained by the following theorem.
\begin{theorem}\label{smallgamma+}
We have the following estimates.

(1) For each $j \ge 1$, $k \ge \log(1/\epsilon)$, we have inequality (\ref{smallgamma1+}) when $da_{1} \neq a_{2}$.

(2) When $a_{1}=1$, $a_{2}=d$ and $c =0$, inequality (\ref{smallgamma1++}) holds for each $k \ge \log(1/\epsilon)$ and $j > mk +\epsilon^{\prime} j$.
\end{theorem}
\begin{proof}
\textbf{(1)} For any complex-valued function $g$ defined on $\mathbb{R}^{2} \times E_{j,k}$, define the operator
\[S_{j,k}g(y,t_{l}) = \sum_{t_{l^{\prime}} \in E_{j,k}} \int_{\mathbb{R}^{2}}  e^{i[y\cdot \xi - c2^{dk}(t_{l}^{a_{2}} -t_{l^{\prime}}^{a_{2}})\xi_{2} -  (\tilde{\Phi}(t_{l},\xi,\delta)-\tilde{\Phi}(t_{l^{\prime}},\xi,\delta))]} a_{j}(t_{l},t_{l^{\prime}},\xi)\hat{g}(\xi, t_{l^{\prime}})d\xi,\]
where $a_{j}(t_{l},t_{l^{\prime}},\xi) = a_{j,k,t_{l}}(\xi) \overline{a_{j,k,t_{l^{\prime}}}(\xi)}$, $\hat{g}(\xi, t_{l^{\prime}})$ means the Fourier transform of $g(y,t_{l^{\prime}})$ with respect to $y$. By duality, it suffices to prove that
\begin{equation}\label{dual1}
\|S_{j,k}g\|_{L^{2(1+2\gamma), \infty}(\mathbb{R}^{2}\times E_{j,k})} \lesssim (c2^{\frac{d\gamma k}{1+2\gamma}} +1) 2^{ -\frac{1-2\gamma}{1+2\gamma}j}  \|g\|_{L^{\frac{2+4\gamma}{1+4\gamma},1}({\mathbb{R}}^2 \times E_{j,k})}.
\end{equation}

For fixed $t_{l}$, we decompose $E_{j,k} =\bigcup_{n=0}^{N} E_{j,k,n}$, where $E_{j,k,0}=\{t_{l^{\prime}} \in E_{j,k}: |t_{l} - t_{l^{\prime}}| \le 2 \cdot 2^{-j}\}$ and $E_{j,k,n}=\{t_{l^{\prime}} \in E_{j,k}: 2^{n} 2^{-j}<|t_{l} - t_{l^{\prime}}| \le 2^{n+1} 2^{-j}\}$ for $n \ge 1$, $N$ is a finite positive number. By inequality (3.4) in \cite{AHRS}, there holds
\begin{equation}\label{number1}
Card(E_{j,k,n}) \lesssim 2^{n\gamma} (c2^{d\gamma k} +1),
\end{equation}
where $Card(E_{j,k,n})$ means the number of the elements in $E_{j,k,n}$. For each $0 \le n \le N$, we define
\[S_{j,k,n}g(y,t_{l}) = \sum_{t_{l^{\prime}} \in E_{j,k,n}} \int_{\mathbb{R}^{2}}  e^{i[y\cdot \xi - c2^{dk}(t_{l}^{a_{2}} -t_{l^{\prime}}^{a_{2}})\xi_{2} - (\tilde{\Phi}(t_{l},\xi,\delta)-\tilde{\Phi}(t_{l^{\prime}},\xi,\delta))]} a_{j}(t_{l},t_{l^{\prime}},\xi)\hat{g}(\xi, t_{l^{\prime}})d\xi.\]
By Bourgain's interpolation technique, inequality (\ref{dual1}) will be established provided that for each $0 \le n \le N$, there holds
\begin{equation}\label{blinearL2}
\|S_{j,k,n}g\|_{L^{2}(\mathbb{R}^{2}\times E_{j,k})} \lesssim (c2^{d\gamma k} +1) 2^{n\gamma} \|g\|_{L^{2}({\mathbb{R}}^2 \times E_{j,k})},
\end{equation}
and
\begin{equation}\label{bilinearlinftyl1}
\|S_{j,k,n}g\|_{L^{\infty}(\mathbb{R}^{2}\times E_{j,k})} \lesssim 2^{-n/2} 2^{2j} \|g\|_{L^{1}({\mathbb{R}}^2 \times E_{j,k})}.
\end{equation}

We first prove (\ref{blinearL2}). By Cauchy-Schwarz inequality, inequality (\ref{number1}) and Plancherel's theorem,
\begin{align}
\|S_{j,k,n}g\|_{L^{2}(\mathbb{R}^{2}\times E_{j,k})}
%&\le \biggl(\sum_{t_{l} \in E_{j,k}} 2^{n\gamma} (c2^{d\gamma k} +1) \sum_{t_{l^{\prime}} \in E_{j,k,n}}\biggl\| \int_{\mathbb{R}^{2}}  e^{i(y\cdot \xi - c2^{dk}(t_{l} - t_{l^{\prime}})\xi_{2} - (t_{l} -t_{l^{\prime}})\xi_2\tilde{\Phi}(s))} a_{j}(t_{\nu},t_{\nu^{\prime}},\xi)\hat{g}(\xi, t_{l^{\prime}})d\xi \biggl\|_{L^{2}(\mathbb{R}^{2})}^{2}\biggl)^{1/2} \nonumber\\
&\le \biggl(\sum_{t_{l} \in E_{j,k}} 2^{ n\gamma} (c2^{ d\gamma k} +1) \sum_{t_{l^{\prime}} \in E_{j,k,n}}\| g(y,t_{l^{\prime}}) \|_{L^{2}(\mathbb{R}^{2})}^{2}\biggl)^{1/2} \nonumber\\
&\le 2^{n\gamma} (c2^{d\gamma k} +1) \|g\|_{L^{2}(\mathbb{R}^{2} \times E_{j,k})},
\end{align}
then we arrive at inequality (\ref{blinearL2}).

In order to prove (\ref{bilinearlinftyl1}), we write
\[S_{j,k,n}g(y,t_{l}) = \sum_{t_{l^{\prime}} \in E_{j,k,n}} \int_{\mathbb{R}^{2}} \int_{\mathbb{R}^{2}}  e^{i[(y-z)\cdot \xi - c2^{dk}(t_{l}^{a_{2}} -t_{l^{\prime}}^{a_{2}})\xi_{2} - (\tilde{\Phi}(t_{l},\xi,\delta)-\tilde{\Phi}(t_{l^{\prime}},\xi,\delta))]} a_{j}(t_{l},t_{l^{\prime}},\xi)d\xi g(z, t_{l^{\prime}})dz. \]
Denote
\[K_{n,t_{l},t_{l^{\prime}}}(y,z) = \int_{\mathbb{R}^{2}}  e^{i[(y-z)\cdot \xi - c2^{dk}(t_{l}^{a_{2}} -t_{l^{\prime}}^{a_{2}})\xi_{2} - (\tilde{\Phi}(t_{l},\xi,\delta)-\tilde{\Phi}(t_{l^{\prime}},\xi,\delta))]} a_{j}(t_{l},t_{l^{\prime}},\xi) d\xi.\]
According to Lemma \ref{kernel estimate3}, we can obtain that for each $t_{l}$, $t_{l^{\prime}}$, $y,z$ and $0 \le n \le N$,  there holds a uniform estimate
\begin{equation}\label{kernel2}
|K_{n,t_{l},t_{l^{\prime}}}(y,z)| \lesssim  2^{2j}(1+2^{j}|t_{l}-t_{l^{\prime}}|)^{-\frac{1}{2}}.
\end{equation}
Then inequality (\ref{bilinearlinftyl1}) follows.

%Since for fixed $t_{l}$ and $t_{l^{\prime}}$, the support of $a(t_{l},t_{l^{\prime}}, \xi)$ is supported in $\{\xi \in \mathbb{R}^{2}: |\xi| \approx 2^{j}\}$, inequality (\ref{kernel2}) is trivial for $0 \le n \lesssim 1$. For $n \gg 1$, by Lemma \ref{kernel1},
%\begin{align}
%|K_{n,t_{l},t_{l^{\prime}}}(y,z)| % &=2^{2j} \biggl|\int_{\mathbb{R}^{2}}  e^{i2^{j}((y-z)\cdot \eta - c2^{dk}(t_{l} - t_{l^{\prime}})\eta_{2} - (t_{l} -t_{l^{\prime}})\eta_2\tilde{\Phi}(s))} a_{j}(t_{\nu},t_{\nu^{\prime}},2^{j}\eta)d\eta \biggl| \nonumber\\
%&=2^{2j} \biggl|\int_{\mathbb{R}^{2}}  e^{i2^{j} |t_{l} - t_{l^{\prime}}| (\frac{y-z}{|t_{l} - t_{l^{\prime}}|}\cdot \eta - c2^{dk}\frac{t_{l} - t_{l^{\prime}}}{|t_{l} - t_{l^{\prime}}|}\eta_{2} - \frac{t_{l} - t_{l^{\prime}}}{|t_{l} - t_{l^{\prime}}|}\eta_2\tilde{\Phi}(s))} a_{j}(t_{l},t_{l^{\prime}},2^{j}\eta)d\eta \biggl| \nonumber\\
%&\lesssim (2^{j}|t_{\nu} - t_{\nu^{\prime}}|)^{-1/2} \lesssim  2^{-n/2}. \nonumber
%\end{align}
%Then we finished the proof of Lemma \ref{smallgamma+}.

\textbf{(2)} As in the proof of (1), for any complex-valued function $g$ defined on $\mathbb{R}^{2} \times \widetilde{E}_{j,k}$, we define the operator
\[\widetilde{S}_{j,k}g(y,t_{l}) = \sum_{t_{l^{\prime}} \in \widetilde{E}_{j,k}} \int_{\mathbb{R}^{2}}  e^{i[y\cdot \xi -  (\tilde{\Phi}(t_{l},\xi,\delta)-\tilde{\Phi}(t_{l^{\prime}},\xi,\delta))]} a_{j}(t_{l},t_{l^{\prime}},\xi)\hat{g}(\xi, t_{l^{\prime}})d\xi.\]
 Let $t_{l}$ be fixed, $\widetilde{E}_{j,k,0}=\{t_{l^{\prime}} \in \widetilde{E}_{j,k}: |t_{l} - t_{l^{\prime}}| \le 2 \cdot 2^{-j+mk}\}$ and $\widetilde{E}_{j,k,n}=\{t_{l^{\prime}} \in \widetilde{E}_{j,k}: 2^{n} 2^{-j+mk}<|t_{l} - t_{l^{\prime}}| \le 2^{n+1} 2^{-j+mk}\}$ for $n \ge 1$. Then $\widetilde{E}_{j,k} =\bigcup_{n=0}^{N} \widetilde{E}_{j,k,n}$,  $N \le j-mk$, and
\begin{equation}\label{number2}
Card(\widetilde{E}_{j,k,n}) \lesssim 2^{n\gamma}.
\end{equation}
We define
\[\widetilde{S}_{j,k,n}g(y,t_{l}) = \sum_{t_{l^{\prime}} \in \widetilde{E}_{j,k,n}} \int_{\mathbb{R}^{2}}  e^{i[y\cdot \xi  - (\tilde{\Phi}(t_{l},\xi,\delta)-\tilde{\Phi}(t_{l^{\prime}},\xi,\delta))]} a_{j}(t_{l},t_{l^{\prime}},\xi)\hat{g}(\xi, t_{l^{\prime}})d\xi.\]
Due to the condition that $a_{1}=1$, $a_{2}=d$,  then
\[\tilde{\Phi}(t_{l},\xi,\delta)-\tilde{\Phi}(t_{l^{\prime}},\xi,\delta)=(t_{l}^{-m }-t_{l^{\prime}}^{- m }) \xi_{2}\frac{\delta^m\phi^{(m)}(0)}{ m!}
\biggl(-\frac{\xi_1}{d\xi_2}\biggl)^{\frac{d+m}{d-1}}+R(t_{l^{\prime}},\xi,\delta,d)- R(t_{l},\xi,\delta,d).\]
Then by Lemma \ref{kernel estimate3}, we have the kernel estimate
\begin{equation}\label{kernel2+}
|K_{n,t_{l},t_{l^{\prime}}}(y,z)| \lesssim  2^{2j}(1+2^{j-mk}|t_{l}-t_{l^{\prime}}|)^{-\frac{1}{2}},
\end{equation}
where
\[K_{n,t_{l},t_{l^{\prime}}}(y,z) = \int_{\mathbb{R}^{2}}  e^{i[(y-z)\cdot \xi - (\tilde{\Phi}(t_{l},\xi,\delta)-\tilde{\Phi}(t_{l^{\prime}},\xi,\delta))]} a_{j}(t_{l},t_{l^{\prime}},\xi) d\xi.\]
Based on (\ref{number2}) and (\ref{kernel2+}), we can use a similar method as in (1) to prove the validity of
\[
\|\widetilde{S}_{j,k,n}g\|_{L^{2}(\mathbb{R}^{2}\times \widetilde{E}_{j,k})} \lesssim  2^{n\gamma} \|g\|_{L^{2}({\mathbb{R}}^2 \times \tilde{E}_{j,k})},
\]
and
\[
\|\widetilde{S}_{j,k,n}g\|_{L^{\infty}(\mathbb{R}^{2}\times \widetilde{E}_{j,k})} \lesssim 2^{-n/2} 2^{2j} \|g\|_{L^{1}({\mathbb{R}}^2 \times \tilde{E}_{j,k})},
\]
and then we can complete the proof by Bourgain's interpolation technique.
\end{proof}

\end{document}